\documentclass[a4paper,oneside,10pt]{article}%
\usepackage{amsmath}
\usepackage{amsfonts}
\usepackage{amssymb}
\usepackage{graphicx}
\usepackage{color}
\usepackage[square,numbers,sort&compress]{natbib}%
\providecommand{\U}[1]{\protect \rule{.1in}{.1in}}

\newtheorem{theorem}{Theorem}[section]

\newtheorem{definition}[theorem]{Definition}

\newtheorem{example}[theorem]{Example}

\newtheorem{proposition}[theorem]{Proposition}
\newtheorem{remark}[theorem]{Remark}

\newenvironment{proof}[1][Proof]{\noindent \textbf{#1.} }{\  $\Box$}
\numberwithin{equation}{section}

\begin{document}

\title{Relationship between MP and DPP for stochastic recursive optimal control
problem under volatility uncertainty }
\author{Xiaojuan Li\thanks{Zhongtai Securities Institute for Financial Studies,
Shandong University, Jinan 250100, China. lixiaojuan@mail.sdu.edu.cn. Research
supported by NSF (No. 11671231).} }
\maketitle

\textbf{Abstract}. In this paper, we study the relationship between maximum
principle (MP) and dynamic programming principle (DPP) for stochastic
recursive optimal control problem driven by $G$-Brownian motion. Under the
smooth assumption for the value function, we obtain the connection between MP
and DPP under a reference probability $P_{t,x}^{\ast}$. Within the framework
of viscosity solution, we establish the relation between the first-order
super-jet, sub-jet of the value function and the solution to the adjoint
equation respectively.

{\textbf{Key words}. } Stochastic recursive optimal control, Maximum
principle, Dynamic programming principle, $G$-expectation

\textbf{AMS subject classifications.} 93E20, 60H10, 35K15

\addcontentsline{toc}{section}{\hspace*{1.8em}Abstract}

\section{Introduction}

Motivated by the study of volatility uncertainty in finance, Peng \cite{P07a,
P08a} established the theory of $G$-expectation which is a consistent
sublinear expectation. The representation of $G$-expectation as the supremum
of expectations over a set of nondominated probability measures $\mathcal{P}$
was obtained in \cite{DHP11, HP09}. Epstein and Ji \cite{EJ1, EJ2} studied the
volatility uncertainty in economic and financial problems by using
$G$-expectation as a tool. Hu et al. \cite{HJPS1, HJPS} obtained the existence
and uniqueness theorem and other properties for backward stochastic
differential equation driven by $G$-Brownian motion ($G$-BSDE), which is
completely different from the classical BSDE due to the set of nondominated
probability measures $\mathcal{P}$ representing $G$-expectation. In addition,
Soner et al. \cite{STZ11} studied a new type of fully nonlinear BSDE, called
$2$BSDE, by different formulation and method, and obtained the deep result of
the existence and uniqueness theorem for $2$BSDE.

Hu and Ji \cite{HJ0, HJ1} studied the maximum principle (MP) and dynamic
programming principle (DPP) for the following stochastic recursive optimal
control problem driven by $G$-Brownian motion:%
\begin{equation}
\left \{
\begin{array}
[c]{rl}%
dX_{s}^{t,x,u}= & b(s,X_{s}^{t,x,u},u_{s})ds+h_{ij}(s,X_{s}^{t,x,u}%
,u_{s})d\langle B^{i},B^{j}\rangle_{s}+\sigma(s,X_{s}^{t,x,u},u_{s})dB_{s},\\
X_{t}^{t,x,u}= & x,\text{ }s\in \lbrack t,T],
\end{array}
\right.  \label{e1-1}%
\end{equation}%
\begin{equation}
\left \{
\begin{array}
[c]{rl}%
dY_{s}^{t,x,u}= & -f(s,X_{s}^{t,x,u},Y_{s}^{t,x,u},Z_{s}^{t,x,u}%
,u_{s})ds-g_{ij}(s,X_{s}^{t,x,u},Y_{s}^{t,x,u},Z_{s}^{t,x,u},u_{s})d\langle
B^{i},B^{j}\rangle_{s}\\
& +Z_{s}^{t,x,u}dB_{s}+dK_{s}^{t,x,u},\\
Y_{T}^{t,x,u}= & \Phi(X_{T}^{t,x,u}),\text{ }s\in \lbrack t,T],
\end{array}
\right.  \label{ee1-1}%
\end{equation}
where $B=(B^{1},\ldots,B^{d})$ is a $d$-dimensional $G$-Brownian motion,
$(\langle B^{i},B^{j}\rangle)_{ij}$ is the quadratic variation of $B$, the set
of all admissible controls $(u_{s})_{s\in \lbrack t,T]}$ is denoted by
$\mathcal{U}^{t}[t,T]$ (see (\ref{new-e-2-1}) for definition). The value
function is defined as%
\begin{equation}
V(t,x):=\inf_{u\in \mathcal{U}^{t}[t,T]}Y_{t}^{t,x,u}. \label{e1-2}%
\end{equation}
$\bar{u}(\cdot)\in \mathcal{U}^{t}[t,T]$ is called an optimal control if
$V(t,x)=Y_{t}^{t,x,\bar{u}}$. Since%
\[%
\begin{array}
[c]{rl}%
Y_{t}^{t,x,u}= & \underset{P\in \mathcal{P}}{\sup}E_{P}\left[  \Phi
(X_{T}^{t,x,u})+\int_{t}^{T}f(s,X_{s}^{t,x,u},Y_{s}^{t,x,u},Z_{s}%
^{t,x,u},u_{s})ds\right. \\
& \left.  +\int_{t}^{T}g_{ij}(s,X_{s}^{t,x,u},Y_{s}^{t,x,u},Z_{s}%
^{t,x,u},u_{s})d\langle B^{i},B^{j}\rangle_{s}\right]  ,
\end{array}
\]
the value function defined in (\ref{e1-2}) is an $\inf \sup$ problem, which is
known as the robust optimal control problem. By introducing a new weak
convergence method under $G$-expectation $\mathbb{\hat{E}}[\cdot]$, Hu and Ji
\cite{HJ0} obtained the MP for the control problem (\ref{e1-1})-(\ref{e1-2})
under a reference probability $P_{t,x}^{\ast}\in \mathcal{P}_{t,x}^{\ast
}\subset \mathcal{P}$ (see (\ref{e2-3-1}) for definition), and proved that this
MP is also a suficient condition under some convex assumptions. In particular,
the adjoint equation (\ref{e2-4}) is introduced under $P_{t,x}^{\ast}$ and has
a unique solution $(p_{s},q_{s},N_{s})_{s\in \lbrack t,T]}$. By introducing a
new implied partition approach, Hu and Ji \cite{HJ1} obtained the DPP and the
related HJB equation. Let us mention that Biagini et al. \cite{BM} and Xu
\cite{Xu} also studied the MP under $G$-expectation framework by different
formulation and method.

It is well-known that MP and DPP are two important methods to study control
problems. The relationship between MP and DPP is studied in many literatures,
we refer the readers to \cite{BCM, HJX, NSW, Shi, J.Yong, Zhang, Z90} and the
references therein. Zhou \cite{Z90} and Yong, Zhou \cite{J.Yong} obtained the
relationship between MP and DPP for stochastic optimal control problems. The
relationship between MP and DPP for different types of stochastic recursive
optimal control problems was established in \cite{HJX, NSW, Shi, Zhang}.
Bahlali et al. \cite{BCM} obtained the relationship between MP and DPP in
singular stochastic control.

In this paper, we study the relationship between MP and DPP for the control
problem (\ref{e1-1})-(\ref{e1-2}). Up to our knowledge, there is no result on
this topic. If the value function $V(\cdot)\in C^{1,2}([t,T]\times \mathbb{R})$
and $\partial_{xx}^{2}V(\cdot)$ is of polynomial growth in $x$, then, for any
$P_{t,x}^{\ast}\in \mathcal{\tilde{P}}_{t,x}^{\ast}\subset \mathcal{P}%
_{t,x}^{\ast}$ (see (\ref{eee3-1}) for definition), we have
\begin{equation}
Y_{s}^{t,x,\bar{u}}=V(s,X_{s}^{t,x,\bar{u}}),\text{ }Z_{s}^{t,x,\bar{u}%
}=\sigma(s,X_{s}^{t,x,\bar{u}},\bar{u}_{s})\partial_{x}V(s,X_{s}^{t,x,\bar{u}%
}),\text{ a.e. }s\in \lbrack t,T],\text{ }P_{t,x}^{\ast}\text{-a.s.}
\label{e1-3}%
\end{equation}
In Example \ref{exa3}, we prove that the relational expression (\ref{e1-3})
does not hold for any $P\not \in \mathcal{\tilde{P}}_{t,x}^{\ast}$. If the
value function $V(\cdot)\in C^{1,3}([t,T]\times \mathbb{R})$, $\partial
_{xx}^{2}V(\cdot)$ is of polynomial growth in $x$ and $\partial_{sx}%
^{2}V(\cdot)$ is continuous, then we obtain that, for some $P_{t,x}^{\ast}%
\in \mathcal{\tilde{P}}_{t,x}^{\ast}$,%
\[
p_{s}=\partial_{x}V(s,X_{s}^{t,x,\bar{u}}),\text{ }q_{s}=\sigma(s,X_{s}%
^{t,x,\bar{u}},\bar{u}_{s})\partial_{xx}^{2}V(s,X_{s}^{t,x,\bar{u}}),\text{
}N_{s}=0,\text{ a.e. }s\in \lbrack t,T],\text{ }P_{t,x}^{\ast}\text{-a.s.,}%
\]
and the MP in \cite{HJ0} holds, where $P_{t,x}^{\ast}$ satisfies the condition
(\ref{eee3-4}). Thus, the relational expression (\ref{e1-3}) explains why the
MP holds under a reference probability $P_{t,x}^{\ast}$. Furthermore, we study
the relationship between MP and DPP in the viscosity sense and obtain%
\begin{equation}
D_{x}^{1,-}V(t,x)\subseteq \lbrack \tilde{p}_{t},\bar{p}_{t}],\text{ }\bar
{p}_{t}\in D_{x}^{1,+}V(t,x)\text{ if }\bar{p}_{t}=\tilde{p}_{t}, \label{e1-4}%
\end{equation}
where $\tilde{p}_{t}$ and $\bar{p}_{t}$ are defined in Theorem \ref{MP-DPP-2}.
In particular, $D_{x}^{1,+}V(t,x)$ may be empty in Example \ref{exa1}. The
relational expression (\ref{e1-4}) is completely new and different from the
classical case.

This paper is organized as follows. We recall some basic results of
$G$-expectation, the MP and DPP for stochastic recursive optimal control
problem driven by $G$-Brownian motion in Section 2. In Section 3, we obtain
the relationship between MP and DPP for stochastic recursive optimal control
problem driven by $G$-Brownian motion.

\section{Preliminaries}

\subsection{$G$-expectation}

In this subsection, we recall some basic notions and results of $G$%
-expectation. The readers may refer to \cite{P2019} for more details.

Let $T>0$ be fixed and let $\Omega_{T}=C_{0}([0,T];\mathbb{R}^{d})$ be the
space of $\mathbb{R}^{d}$-valued continuous functions on $[0,T]$ with
$\omega_{0}=0$. The canonical process $B_{t}(\omega):=\omega_{t}$, for
$\omega \in \Omega_{T}$ and $t\in \lbrack0,T]$. For each given $t\in \lbrack0,T]$,
set%
\[
Lip(\Omega_{t}):=\{ \varphi(B_{t_{1}},B_{t_{2}}-B_{t_{1}},\ldots,B_{t_{N}%
}-B_{t_{N-1}}):N\geq1,t_{1}\leq \cdots \leq t_{N}\leq t,\varphi \in
C_{b.Lip}(\mathbb{R}^{d\times N})\},
\]
where $C_{b.Lip}(\mathbb{R}^{d\times N})$ denotes the space of bounded
Lipschitz functions on $\mathbb{R}^{d\times N}$.

Let $G:\mathbb{S}_{d}\rightarrow \mathbb{R}$ be a given monotonic and sublinear
function, where $\mathbb{S}_{d}$ denotes the set of $d\times d$ symmetric
matrices. Then there exists a bounded and convex set $\Sigma \subset
\mathbb{S}_{d}^{+}$ such that%
\[
G(A)=\frac{1}{2}\sup_{\gamma \in \Sigma}\mathrm{tr}[A\gamma]\text{ for }%
A\in \mathbb{S}_{d},
\]
where $\mathbb{S}_{d}^{+}$ denotes the set of $d\times d$ nonnegative
matrices. In this paper, we consider the non-degenerate $G$, i.e., there
exists a $\alpha>0$ such that $\gamma \geq \alpha I_{d}$ for any $\gamma
\in \Sigma$. Specially, if $d=1$, then $G(a)=\frac{1}{2}(\bar{\sigma}^{2}%
a^{+}-\underline{\sigma}^{2}a^{-})$ for $a\in \mathbb{R}$ with $\bar{\sigma
}\geq \underline{\sigma}>0$.

Peng \cite{P07a, P08a} constructed the $G$-expectation $\mathbb{\hat{E}%
}:Lip(\Omega_{T})\rightarrow \mathbb{R}$ and the conditional $G$-expectation
$\mathbb{\hat{E}}_{t}:Lip(\Omega_{T})\rightarrow Lip(\Omega_{t})$ as follows:

\begin{description}
\item[(i)] For $s\leq t\leq T$ and $\varphi \in C_{b.Lip}(\mathbb{R}^{d})$,
define $\mathbb{\hat{E}}[\varphi(B_{t}-B_{s})]=u(t-s,0)$, where $u$ is the
viscosity solution (see \cite{CIP}) of the following $G$-heat equation:%
\[
\partial_{t}u-G(\partial_{xx}^{2}u)=0,\ u(0,x)=\varphi(x).
\]

\item[(ii)] For $X=\varphi(B_{t_{1}},B_{t_{2}}-B_{t_{1}},\ldots,B_{t_{N}%
}-B_{t_{N-1}})\in Lip(\Omega_{T})$, define
\[
\mathbb{\hat{E}}_{t_{i}}[X]=\varphi_{i}(B_{t_{1}},\ldots,B_{t_{i}}-B_{t_{i-1}%
})\text{ for }i=N-1,\ldots,1\text{ and }\mathbb{\hat{E}}[X]=\mathbb{\hat{E}%
}[\varphi_{1}(B_{t_{1}})],
\]
where $\varphi_{N-1}(x_{1},\ldots,x_{N-1}):=\mathbb{\hat{E}}[\varphi
(x_{1},\ldots,x_{N-1},B_{t_{N}}-B_{t_{N-1}})]$ and
\[
\varphi_{i}(x_{1},\ldots,x_{i}):=\mathbb{\hat{E}}[\varphi_{i+1}(x_{1}%
,\ldots,x_{i},B_{t_{i+1}}-B_{t_{i}})]\text{ for }i=N-2,\ldots,1.
\]

\end{description}

The space $(\Omega_{T},Lip(\Omega_{T}),\mathbb{\hat{E}},(\mathbb{\hat{E}}%
_{t})_{t\in \lbrack0,T]})$ is a consistent sublinear expectation space, where
$\mathbb{\hat{E}}_{0}=\mathbb{\hat{E}}$. The canonical process $(B_{t}%
)_{t\in \lbrack0,T]}$ is called the $G$-Brownian motion under $\mathbb{\hat{E}%
}$.

For each $t\in \lbrack0,T]$, denote by $L_{G}^{p}(\Omega_{t})$ the completion
of $Lip(\Omega_{t})$ under the norm $||X||_{L_{G}^{p}}:=(\mathbb{\hat{E}%
}[|X|^{p}])^{1/p}$ for $p\geq1$. $\mathbb{\hat{E}}_{t}$ can be continuously
extended to $L_{G}^{1}(\Omega_{T})$ under the norm $||\cdot||_{L_{G}^{1}}$. A
process $(M_{t})_{t\in \lbrack0,T]}$ is called a $G$-martingale if $M_{T}\in
L_{G}^{1}(\Omega_{T})$ and $M_{t}=\mathbb{\hat{E}}_{t}[M_{T}]$.

\begin{theorem}
\label{re-1}(\cite{DHP11, HP09}) There exists a unique convex and weakly
compact set of probability measures $\mathcal{P}$ on $(\Omega_{T}%
,\mathcal{B}(\Omega_{T}))$ such that%
\[
\mathbb{\hat{E}}[X]=\sup_{P\in \mathcal{P}}E_{P}[X]\text{ for all }X\in
L_{G}^{1}(\Omega_{T}).
\]
$\mathcal{P}$ is called a set that represents $\mathbb{\hat{E}}$.
\end{theorem}

For this $\mathcal{P}$, we define capacity%
\[
c(A):=\sup_{P\in \mathcal{P}}P(A)\text{ for }A\in \mathcal{B}(\Omega_{T}).
\]
A set $A\in \mathcal{B}(\Omega_{T})$ is polar if $c(A)=0$. A property holds
\textquotedblleft quasi-surely" (q.s. for short) if it holds outside a polar
set. In the following, we do not distinguish two random variables $X$ and $Y$
if $X=Y$ q.s.

\begin{definition}
Let $M_{G}^{0}(0,T)$ be the space of simple processes in the following form:
for each $N\in \mathbb{N}$ and $0=t_{0}<\cdots<t_{N}=T$,%
\[
\eta_{t}=\sum_{i=0}^{N-1}\xi_{i}I_{[t_{i},t_{i+1})}(t),
\]
where $\xi_{i}\in Lip(\Omega_{t_{i}})$ for $i=0,1,\ldots,N-1$.
\end{definition}

Denote by $M_{G}^{p}(0,T)$ the completion of $M_{G}^{0}(0,T)$ under the norm
$||\eta||_{M_{G}^{p}}:=\left(  \mathbb{\hat{E}}[\int_{0}^{T}|\eta_{t}%
|^{p}dt]\right)  ^{1/p}$ for $p\geq1$. For each $\eta^{i}\in M_{G}^{2}(0,T)$,
$i=1,\ldots,d$, denote $\eta=(\eta^{1},\ldots,\eta^{d})^{T}\in M_{G}%
^{2}(0,T;\mathbb{R}^{d})$, the $G$-It\^{o} integral $\int_{0}^{T}\eta_{t}%
^{T}dB_{t}$ is well defined.

\subsection{MP and DPP under volatility uncertainty}

For each given $t\leq s\leq T$, set%
\[
Lip(\Omega_{s}^{t}):=\{ \varphi(B_{t_{1}}-B_{t},\ldots,B_{t_{N}}-B_{t}%
):N\geq1,t\leq t_{1}\leq \cdots \leq t_{N}\leq s,\varphi \in C_{b.Lip}%
(\mathbb{R}^{d\times N})\},
\]%
\[
M_{G}^{0,t}(t,T)=\left \{  \eta_{s}=\sum_{i=0}^{N-1}\xi_{i}I_{[t_{i},t_{i+1}%
)}(s):N\geq1,t=t_{0}\leq \cdots \leq t_{N}=T,\xi_{i}\in Lip(\Omega_{t_{i}}%
^{t})\right \}  .
\]
Denote by $M_{G}^{p,t}(t,T)$ (resp. $L_{G}^{p}(\Omega_{s}^{t})$) the
completion of $M_{G}^{0,t}(t,T)$ (resp. $Lip(\Omega_{s}^{t})$) under the norm
$||\eta||_{M_{G}^{p}}:=\left(  \mathbb{\hat{E}}[\int_{t}^{T}|\eta_{s}%
|^{p}ds]\right)  ^{1/p}$ (resp. $||X||_{L_{G}^{p}}:=(\mathbb{\hat{E}}%
[|X|^{p}])^{1/p}$) for $p\geq1$.

For simplicity of presentation, we only consider the relationship between MP
and DPP under $1$-dimensional case studied in \cite{HJ0, HJ1}. The results
still hold for the general case (\ref{e1-1})-(\ref{e1-2}). More precisely, let
$U\subset \mathbb{R}^{m}$ be a given nonempty convex and compact set, consider
the following $1$-dimensional forward and backward SDEs driven by
$1$-dimensional $G$-Brownian motion $B$ for each $(t,x)\in \lbrack
0,T]\times \mathbb{R}$:%
\begin{equation}
\left \{
\begin{array}
[c]{rl}%
dX_{s}^{t,x,u}= & h(s,X_{s}^{t,x,u},u_{s})d\langle B\rangle_{s}+\sigma
(s,X_{s}^{t,x,u},u_{s})dB_{s},\\
dY_{s}^{t,x,u}= & -g(s,X_{s}^{t,x,u},Y_{s}^{t,x,u},Z_{s}^{t,x,u}%
,u_{s})d\langle B\rangle_{s}+Z_{s}^{t,x,u}dB_{s}+dK_{s}^{t,x,u},\\
X_{t}^{t,x,u}= & x,\text{ }Y_{T}^{t,x,u}=\Phi(X_{T}^{t,x,u}),\text{ }%
s\in \lbrack t,T],
\end{array}
\right.  \label{e2-1}%
\end{equation}
where $\langle B\rangle$ is the quadratic variation of $B$, $h$,
$\sigma:[0,T]\times \mathbb{R}\times U\rightarrow \mathbb{R}$, $g:[0,T]\times
\mathbb{R}\times \mathbb{R}\times \mathbb{R}\times U\rightarrow \mathbb{R}$,
$\Phi:\mathbb{R}\rightarrow \mathbb{R}$ are deterministic and continuous
functions satisfying the following conditions:

\begin{description}
\item[(H1)] The derivatives of $h$, $\sigma$, $g$, $\Phi$ in $(x,y,z,u)$ are
continuous in $(s,x,y,z,u)$.

\item[(H2)] There exists a constant $L>0$ such that for any $(s,x,y,z,u)\in
\lbrack0,T]\times \mathbb{R}\times \mathbb{R}\times \mathbb{R}\times U$,%
\[%
\begin{array}
[c]{l}%
|h_{x}(s,x,u)|+|h_{u}(s,x,u)|+|\sigma_{x}(s,x,u)|+|\sigma_{u}(s,x,u)|+|g_{y}%
(s,x,y,z,u)|+|g_{z}(s,x,y,z,u)|\leq L,\\
|g_{x}(s,x,y,z,u)|+|g_{u}(s,x,y,z,u)|+|\Phi^{\prime}(x)|\leq L(1+|x|+|u|).
\end{array}
\]

\end{description}

\begin{remark}
Here $G(a)=\frac{1}{2}(\bar{\sigma}^{2}a^{+}-\underline{\sigma}^{2}a^{-})$ for
$a\in \mathbb{R}$ with $\bar{\sigma}\geq \underline{\sigma}>0$. By Corollary
3.5.5 in \cite{P2019}, we know $d\langle B\rangle_{s}=\gamma_{s}ds$ with
$\underline{\sigma}^{2}\leq \gamma_{s}\leq \bar{\sigma}^{2}$.
\end{remark}

For each $t\in \lbrack0,T]$, we denote by
\begin{equation}
\mathcal{U}^{t}[t,T]:=\{u(\cdot):u(\cdot)\in M_{G}^{2,t}(t,T;\mathbb{R}%
^{m})\text{ with values in }U\} \label{new-e-2-1}%
\end{equation}
the set of admissible controls on $[t,T]$. For each $u(\cdot)\in
\mathcal{U}^{t}[t,T]$, the equation (\ref{e2-1}) has a unique solution
$(X^{t,x,u},Y^{t,x,u},Z^{t,x,u},K^{t,x,u})$ (see \cite{HJPS1}) such that%
\begin{equation}
\mathbb{\hat{E}}\left[  \sup_{s\in \lbrack t,T]}(|X_{s}^{t,x,u}|^{p}%
+|Y_{s}^{t,x,u}|^{p})+\left(  \int_{t}^{T}|Z_{s}^{t,x,u}|^{2}ds\right)
^{p/2}+|K_{T}^{t,x,u}|^{p}\right]  <\infty \text{ for each }p\geq2
\label{ee2-1}%
\end{equation}
and%
\[
X_{s}^{t,x,u},\text{ }Y_{s}^{t,x,u},\text{ }K_{s}^{t,x,u}\in L_{G}^{p}%
(\Omega_{s}^{t}),\text{ }K^{t,x,u}\text{ is a non-increasing }%
G\text{-martingale with }K_{t}^{t,x,u}=0.
\]
Then $Y_{t}^{t,x,u}\in L_{G}^{p}(\Omega_{t}^{t})=\mathbb{R}$. For each fixed
$(t,x)\in \lbrack0,T]\times \mathbb{R}$, we define the cost functional
\begin{equation}
J(t,x;u(\cdot))=Y_{t}^{t,x,u} \label{e2-2}%
\end{equation}
and the value function%
\begin{equation}
V(t,x)=\inf_{u\in \mathcal{U}^{t}[t,T]}J(t,x;u(\cdot)). \label{e2-3}%
\end{equation}
Let $\bar{u}(\cdot)\in \mathcal{U}^{t}[t,T]$ be an optimal control. Then
$V(t,x)=J(t,x;\bar{u}(\cdot))$. The corresponding solution $(X^{t,x,\bar{u}%
},Y^{t,x,\bar{u}},Z^{t,x,\bar{u}},K^{t,x,\bar{u}})$ to equation (\ref{e2-1})
is called the optimal trajectory. Set%
\begin{equation}
\mathcal{P}_{t,x}^{\ast}=\{P\in \mathcal{P}:E_{P}[K_{T}^{t,x,\bar{u}}]=0\}.
\label{e2-3-1}%
\end{equation}

\begin{remark}
Note that $K_{T}^{t,x,\bar{u}}\leq0$, then $E_{P}[K_{T}^{t,x,\bar{u}}]=0$ iff
$K_{T}^{t,x,\bar{u}}=0$ under $P$.
\end{remark}

For each fixed $P_{t,x}^{\ast}\in \mathcal{P}_{t,x}^{\ast}$, the following
adjoint equation%
\begin{equation}
\left \{
\begin{array}
[c]{rl}%
dp_{s}= & -\{[h_{x}(s)+g_{y}(s)+g_{z}(s)\sigma_{x}(s)]p_{s}+[g_{z}%
(s)+\sigma_{x}(s)]q_{s}+g_{x}(s)\}d\langle B\rangle_{s}\\
& +q_{s}dB_{s}+dN_{s},\\
p_{T}= & \Phi^{\prime}(X_{T}^{t,x,\bar{u}}),\text{ }s\in \lbrack t,T],
\end{array}
\right.  \label{e2-4}%
\end{equation}
has a unique solution $(p(\cdot),q(\cdot),N(\cdot))\in M_{P_{t,x}^{\ast}%
}^{2,t}(t,T)\times M_{P_{t,x}^{\ast}}^{2,t}(t,T)\times M_{P_{t,x}^{\ast}%
}^{2,t,\bot}(t,T)$ under $P_{t,x}^{\ast}$ (see \cite{BL, EH}), where
$h_{x}(s)=h_{x}(s,X_{s}^{t,x,\bar{u}},\bar{u}_{s})$, similar for $\sigma
_{x}(s)$, $g_{x}(s)$, $g_{y}(s)$ and $g_{z}(s)$, $\mathcal{F}_{s}^{t}%
=\sigma(B_{r}-B_{t}:t\leq r\leq s)$,
\[
M_{P_{t,x}^{\ast}}^{2,t}(t,T)=\left \{  (\eta_{s})_{s\in \lbrack t,T]}:\eta
_{s}\in \mathcal{F}_{s}^{t}\text{ and }E_{P_{t,x}^{\ast}}\left[  \int_{t}%
^{T}|\eta_{s}|^{2}ds\right]  <\infty \right \}  ,
\]%
\[%
\begin{array}
[c]{cl}%
M_{P_{t,x}^{\ast}}^{2,t,\bot}(t,T)= & \left \{  (N_{s})_{s\in \lbrack
t,T]}:N_{t}=0\text{, }N_{s}\in \mathcal{F}_{s}^{t}\text{, }N\text{ is a
square}\right. \\
& \left.  \text{integrable martingale that is orthogonal to }B\right \}  .
\end{array}
\]
Define the Hamiltonian $H:\mathbb{R}\times \mathbb{R}\times \mathbb{R}\times
U\times U\times \mathbb{R}\times \mathbb{R}\times \lbrack t,T]\rightarrow
\mathbb{R}$ as follows:%
\begin{equation}
H(x,y,z,u,v,p,q,s)=h(s,x,u)p+\sigma(s,x,u)q+g_{z}(s,x,y,z,v)\sigma
(s,x,u)p+g(s,x,y,z,u). \label{e2-5}%
\end{equation}
Hu and Ji \cite{HJ0} obtained the following MP.

\begin{theorem}
\label{MP}(\cite{HJ0}) Suppose that Assumptions (H1) and (H2) hold. Let
$\bar{u}(\cdot)\in \mathcal{U}^{t}[t,T]$ be an optimal control for the control
problem (\ref{e2-3}) and let $(X^{t,x,\bar{u}},Y^{t,x,\bar{u}},Z^{t,x,\bar{u}%
},K^{t,x,\bar{u}})$ be the corresponding trajectory. Then there exists a
$P_{t,x}^{\ast}\in \mathcal{P}_{t,x}^{\ast}$ such that the following maximum
principle holds:%
\[
\langle H_{u}(X_{s}^{t,x,\bar{u}},Y_{s}^{t,x,\bar{u}},Z_{s}^{t,x,\bar{u}}%
,\bar{u}_{s},\bar{u}_{s},p_{s},q_{s},s),u-\bar{u}_{s}\rangle \geq0\text{,
}\forall u\in U\text{, a.e. }s\in \lbrack t,T]\text{, }P_{t,x}^{\ast
}\text{-a.s.,}%
\]
where $(p(\cdot),q(\cdot),N(\cdot))$ is the solution of the adjoint equation
(\ref{e2-4}) under $P_{t,x}^{\ast}$, $H(\cdot)$ is defined in (\ref{e2-5}).
\end{theorem}

Hu and Ji \cite{HJ1} obtained that the value function $V(\cdot)$ satisfies the
DPP. Based on DPP, they obtained the following theorem.

\begin{theorem}
\label{th-DPP}(\cite{HJ1}) Suppose that Assumptions (H1) and (H2) hold. Let
$V(\cdot)$ be the value function defined in (\ref{e2-3}). Then $V(\cdot)$ is
the unique viscosity solution of the following second-order partial
differential equation:%
\begin{equation}
\partial_{t}V(t,x)+\inf_{u\in U}G(F(t,x,V(t,x),\partial_{x}V(t,x),\partial
_{xx}^{2}V(t,x),u))=0,\text{ }V(T,x)=\Phi(x), \label{e2-6}%
\end{equation}
where $F:[0,T]\times \mathbb{R}\times \mathbb{R}\times \mathbb{R}\times
\mathbb{R}\times U\rightarrow \mathbb{R}$ defined by%
\[
F(t,x,a_{1},a_{2},a_{3},u)=\sigma^{2}(t,x,u)a_{3}+2h(t,x,u)a_{2}%
+2g(t,x,a_{1},\sigma(t,x,u)a_{2},u).
\]

\end{theorem}

\section{Main results}

In the following, we first study the relationship between MP and DPP under the
smooth case.

\begin{theorem}
\label{MP-DPP-1}Suppose that Assumptions (H1) and (H2) hold. Let $\bar
{u}(\cdot)\in \mathcal{U}^{t}[t,T]$ be an optimal control for the control
problem (\ref{e2-3}) and let $(X^{t,x,\bar{u}},Y^{t,x,\bar{u}},Z^{t,x,\bar{u}%
},K^{t,x,\bar{u}})$ be the corresponding trajectory.

\begin{description}
\item[(1)] If the value function $V(\cdot)\in C^{1,2}([t,T]\times \mathbb{R})$
and $\partial_{xx}^{2}V(\cdot)$ is of polynomial growth in $x$, then
\begin{equation}
\mathcal{\tilde{P}}_{t,x}^{\ast}=\{P\in \mathcal{P}:E_{P}[\tilde{K}_{T}]=0\}
\label{eee3-1}%
\end{equation}
is nonempty and belongs to $\mathcal{P}_{t,x}^{\ast}$, where $\mathcal{P}%
_{t,x}^{\ast}$ is defined in (\ref{e2-3-1}),%
\begin{equation}%
\begin{array}
[c]{l}%
\tilde{K}_{s}=\frac{1}{2}\int_{t}^{s}F(r)d\langle B\rangle_{r}-\int_{t}%
^{s}G(F(r))dr\text{, }s\in \lbrack t,T],\\
F(s)=F(s,X_{s}^{t,x,\bar{u}},V(s,X_{s}^{t,x,\bar{u}}),\partial_{x}%
V(s,X_{s}^{t,x,\bar{u}}),\partial_{xx}^{2}V(s,X_{s}^{t,x,\bar{u}}),\bar{u}%
_{s}).
\end{array}
\label{eee3-2}%
\end{equation}
Moreover for any $P_{t,x}^{\ast}\in \mathcal{\tilde{P}}_{t,x}^{\ast}$, we have,
$P_{t,x}^{\ast}$-a.s., for a.e. $s\in \lbrack t,T]$,%
\begin{equation}%
\begin{array}
[c]{rl}%
Y_{s}^{t,x,\bar{u}}= & V(s,X_{s}^{t,x,\bar{u}}),\\
Z_{s}^{t,x,\bar{u}}= & \sigma(s,X_{s}^{t,x,\bar{u}},\bar{u}_{s})\partial
_{x}V(s,X_{s}^{t,x,\bar{u}}),
\end{array}
\label{ee3-1}%
\end{equation}
and%
\begin{equation}%
\begin{array}
[c]{l}%
-\partial_{s}V(s,X_{s}^{t,x,\bar{u}})\\
=G(F(s,X_{s}^{t,x,\bar{u}},V(s,X_{s}^{t,x,\bar{u}}),\partial_{x}%
V(s,X_{s}^{t,x,\bar{u}}),\partial_{xx}^{2}V(s,X_{s}^{t,x,\bar{u}}),\bar{u}%
_{s}))\\
=\underset{u\in U}{\min}G(F(s,X_{s}^{t,x,\bar{u}},V(s,X_{s}^{t,x,\bar{u}%
}),\partial_{x}V(s,X_{s}^{t,x,\bar{u}}),\partial_{xx}^{2}V(s,X_{s}%
^{t,x,\bar{u}}),u)).
\end{array}
\label{ee3-2}%
\end{equation}

\item[(2)] If the value function $V(\cdot)\in C^{1,2}([t,T]\times \mathbb{R})$,
$\partial_{xx}^{2}V(\cdot)$ is of polynomial growth in $x$ and $\partial
_{sx}^{2}V(\cdot)$ is continuous, then for any $P_{t,x}^{\ast}\in
\mathcal{\tilde{P}}_{t,x}^{\ast}$, we have, $P_{t,x}^{\ast}$-a.s., for a.e.
$s\in \lbrack t,T]$,%
\begin{equation}%
\begin{array}
[c]{l}%
\partial_{sx}^{2}V(s,X_{s}^{t,x,\bar{u}})=-\frac{1}{2}\gamma_{s}\partial
_{x}F(s)\text{ if }F(s)\not =0,\\
\partial_{sx}^{2}V(s,X_{s}^{t,x,\bar{u}})=-\frac{1}{2}v_{s}\partial
_{x}F(s)\text{ if }F(s)=0,
\end{array}
\label{eee3-3}%
\end{equation}
where $d\langle B\rangle_{s}=\gamma_{s}ds$, $v_{s}=-2\partial_{sx}%
^{2}V(s,X_{s}^{t,x,\bar{u}})(\partial_{x}F(s))^{-1}I_{\{ \partial
_{x}F(s)\not =0\}}+\gamma_{s}I_{\{ \partial_{x}F(s)=0\}}\in \lbrack
\underline{\sigma}^{2},\bar{\sigma}^{2}]$, $P_{t,x}^{\ast}$-a.s., and%
\begin{equation}
\partial_{x}F(s)=\partial_{x}F(s,X_{s}^{t,x,\bar{u}},V(s,X_{s}^{t,x,\bar{u}%
}),\partial_{x}V(s,X_{s}^{t,x,\bar{u}}),\partial_{xx}^{2}V(s,X_{s}%
^{t,x,\bar{u}}),\bar{u}_{s}). \label{eee3-5}%
\end{equation}

\item[(3)] If the value function $V(\cdot)\in C^{1,3}([t,T]\times \mathbb{R})$,
$\partial_{xx}^{2}V(\cdot)$ is of polynomial growth in $x$, $\partial_{sx}%
^{2}V(\cdot)$ is continuous and there exists a $P_{t,x}^{\ast}\in
\mathcal{\tilde{P}}_{t,x}^{\ast}$ such that, $P_{t,x}^{\ast}$-a.s., for a.e.
$s\in \lbrack t,T]$,%
\begin{equation}
\partial_{sx}^{2}V(s,X_{s}^{t,x,\bar{u}})=-\frac{1}{2}\gamma_{s}\partial
_{x}F(s)\text{ if }F(s)=0, \label{eee3-4}%
\end{equation}
then we have, $P_{t,x}^{\ast}$-a.s., for a.e. $s\in \lbrack t,T]$,%
\begin{equation}%
\begin{array}
[c]{rl}%
p_{s}= & \partial_{x}V(s,X_{s}^{t,x,\bar{u}}),\\
q_{s}= & \sigma(s,X_{s}^{t,x,\bar{u}},\bar{u}_{s})\partial_{xx}^{2}%
V(s,X_{s}^{t,x,\bar{u}}),\\
N_{s}= & 0,
\end{array}
\label{ee3-3}%
\end{equation}
and%
\begin{equation}
\langle H_{u}(X_{s}^{t,x,\bar{u}},Y_{s}^{t,x,\bar{u}},Z_{s}^{t,x,\bar{u}}%
,\bar{u}_{s},\bar{u}_{s},p_{s},q_{s},s),u-\bar{u}_{s}\rangle \geq0\text{,
}\forall u\in U\text{,} \label{ee3-4}%
\end{equation}
where $(p(\cdot),q(\cdot),N(\cdot))$ is the solution of the adjoint equation
(\ref{e2-4}) under $P_{t,x}^{\ast}$, $H(\cdot)$ is defined in (\ref{e2-5}).
\end{description}
\end{theorem}

\begin{proof}
(1) Applying Ito's formula to $V(s,X_{s}^{t,x,\bar{u}})$, we have%
\begin{equation}
\left \{
\begin{array}
[c]{rl}%
d\tilde{Y}_{s}= & \Pi_{s}ds-g(s,X_{s}^{t,x,\bar{u}},\tilde{Y}_{s},\tilde
{Z}_{s},\bar{u}_{s})d\langle B\rangle_{s}+\tilde{Z}_{s}dB_{s}+d\tilde{K}%
_{s},\\
\tilde{Y}_{T}= & \Phi(X_{T}^{t,x,\bar{u}}),\text{ }s\in \lbrack t,T],
\end{array}
\right.  \label{ee3-5}%
\end{equation}
where $\tilde{K}_{s}$ and $F(s)$ are defined in (\ref{eee3-2}),
\[
\tilde{Y}_{s}=V(s,X_{s}^{t,x,\bar{u}}),\text{ }\tilde{Z}_{s}=\sigma
(s,X_{s}^{t,x,\bar{u}},\bar{u}_{s})\partial_{x}V(s,X_{s}^{t,x,\bar{u}}),\text{
}\Pi_{s}=\partial_{s}V(s,X_{s}^{t,x,\bar{u}})+G(F(s)).
\]
By (\ref{ee2-1}) and the assumptions of $V(\cdot)$, we can obtain%
\[
\mathbb{\hat{E}}\left[  \sup_{s\in \lbrack t,T]}|\tilde{Y}_{s}|^{2}+\int
_{t}^{T}(|\tilde{Z}_{s}|^{2}+|F(s)|^{2})ds\right]  <\infty.
\]
It follows from Proposition 4.1.4 in \cite{P2019} that $(\tilde{K}_{s}%
)_{s\in \lbrack t,T]}$ is a non-increasing $G$-martingale with $\tilde{K}%
_{t}=0$ and $\mathbb{\hat{E}}\left[  |\tilde{K}_{T}|^{2}\right]  <\infty$.
Noting that $V(\cdot)$ is a solution to PDE (\ref{e2-6}), we obtain $\Pi
_{s}\geq0$. On the other hand, $(Y^{t,x,\bar{u}},Z^{t,x,\bar{u}}%
,K^{t,x,\bar{u}})$ satisfies%
\begin{equation}
\left \{
\begin{array}
[c]{rl}%
dY_{s}^{t,x,\bar{u}}= & -g(s,X_{s}^{t,x,\bar{u}},Y_{s}^{t,x,\bar{u}}%
,Z_{s}^{t,x,\bar{u}},\bar{u}_{s})d\langle B\rangle_{s}+Z_{s}^{t,x,\bar{u}%
}dB_{s}+dK_{s}^{t,x,\bar{u}},\\
Y_{T}^{t,x,\bar{u}}= & \Phi(X_{T}^{t,x,\bar{u}}),\text{ }s\in \lbrack t,T].
\end{array}
\right.  \label{ee3-6}%
\end{equation}
By $\mathbb{\hat{E}}\left[  \tilde{K}_{T}\right]  =0$ and Theorem \ref{re-1},
it is easy to get that $\mathcal{\tilde{P}}_{t,x}^{\ast}$ defined in
(\ref{eee3-1}) is nonempty. For any given $P_{t,x}^{\ast}\in \mathcal{\tilde
{P}}_{t,x}^{\ast}$, since%
\[
dK_{s}^{t,x,\bar{u}}-d\tilde{K}_{s}-\Pi_{s}ds=dK_{s}^{t,x,\bar{u}}-\Pi
_{s}ds\leq0\text{ for }s\in \lbrack t,T]\text{, }P_{t,x}^{\ast}\text{-a.s.,}%
\]
and $\tilde{Y}_{t}=V(t,x)=Y_{t}^{t,x,\bar{u}}$, by the strict comparison
theorem for BSDEs (\ref{ee3-5}) and (\ref{ee3-6}) under $P_{t,x}^{\ast}$ (see
Theorem 2.2 in \cite{EPQ}), we obtain%
\[
dK_{s}^{t,x,\bar{u}}-\Pi_{s}ds=0\text{ for }s\in \lbrack t,T],\text{ }%
P_{t,x}^{\ast}\text{-a.s.,}%
\]
which implies $K_{T}^{t,x,\bar{u}}=0$ and $\Pi_{s}=0$ for $s\in \lbrack t,T],$
$P_{t,x}^{\ast}$-a.s. Thus $P_{t,x}^{\ast}\in \mathcal{P}_{t,x}^{\ast}$ and
(\ref{ee3-1}), (\ref{ee3-2}) hold under $P_{t,x}^{\ast}$.

(2) By (\ref{e2-6}), we know%
\begin{equation}
\partial_{s}V(s,x)+G(F(s,x,V(s,x),\partial_{x}V(s,x),\partial_{xx}%
^{2}V(s,x),\bar{u}_{s}))\geq0. \label{ee3-7}%
\end{equation}
For any given $P_{t,x}^{\ast}\in \mathcal{\tilde{P}}_{t,x}^{\ast}$, by
(\ref{ee3-2}) we know, $P_{t,x}^{\ast}$-a.s., for a.e. $s\in \lbrack t,T]$,%
\begin{equation}
\partial_{s}V(s,X_{s}^{t,x,\bar{u}})+G(F(s))=0. \label{ee3-8}%
\end{equation}
Under the case $F(s)>0$, noting that $F(s,x,V(s,x),\partial_{x}V(s,x),\partial
_{xx}^{2}V(s,x),\bar{u}_{s})$ is continuous in $x$, we obtain that, for
$x\rightarrow X_{s}^{t,x,\bar{u}}$,%
\begin{align*}
&  G(F(s,x,V(s,x),\partial_{x}V(s,x),\partial_{xx}^{2}V(s,x),\bar{u}%
_{s}))-G(F(s))\\
&  =\frac{1}{2}\bar{\sigma}^{2}[F(s,x,V(s,x),\partial_{x}V(s,x),\partial
_{xx}^{2}V(s,x),\bar{u}_{s})-F(s)],
\end{align*}
which implies $\partial_{sx}^{2}V(s,X_{s}^{t,x,\bar{u}})+\frac{1}{2}%
\bar{\sigma}^{2}\partial_{x}F(s)=0$. Since%
\begin{align*}
\tilde{K}_{T}  &  =\frac{1}{2}\int_{t}^{T}F(s)d\langle B\rangle_{s}-\int
_{t}^{T}G(F(s))ds\\
&  =\frac{1}{2}\int_{t}^{T}(F(s))^{+}(\gamma_{s}-\bar{\sigma}^{2})ds+\frac
{1}{2}\int_{t}^{T}(F(s))^{-}(\underline{\sigma}^{2}-\gamma_{s})ds\\
&  =0,\text{ }P_{t,x}^{\ast}\text{-a.s.,}%
\end{align*}
and $\underline{\sigma}^{2}\leq \gamma_{s}\leq \bar{\sigma}^{2}$, we get
$\gamma_{s}=\bar{\sigma}^{2}$ if $F(s)>0$. Thus we have $\partial_{sx}%
^{2}V(s,X_{s}^{t,x,\bar{u}})=-\frac{1}{2}\gamma_{s}\partial_{x}F(s)$ under the
case $F(s)>0$. By the same method, we can get $\partial_{sx}^{2}%
V(s,X_{s}^{t,x,\bar{u}})=-\frac{1}{2}\gamma_{s}\partial_{x}F(s)$ under the
case $F(s)<0$.

Under the case $F(s)=0$ and $\partial_{x}F(s)=0$, it is clear that
$\partial_{x}|F(s)|=0$. Since $|G(a)|\leq \frac{1}{2}\bar{\sigma}^{2}|a|$ for
$a\in \mathbb{R}$, we have $\partial_{x}G(F(s))=0$. Thus $\partial_{sx}%
^{2}V(s,X_{s}^{t,x,\bar{u}})=0$ by (\ref{ee3-7}) and (\ref{ee3-8}), which
implies $\partial_{sx}^{2}V(s,X_{s}^{t,x,\bar{u}})=-\frac{1}{2}\gamma
_{s}\partial_{x}F(s)$. Under the case $F(s)=0$ and $\partial_{x}F(s)>0$, it is
obvious that
\[
F(s,x,V(s,x),\partial_{x}V(s,x),\partial_{xx}^{2}V(s,x),\bar{u}_{s})>0\text{
for }x\downarrow X_{s}^{t,x,\bar{u}}%
\]
and
\[
F(s,x,V(s,x),\partial_{x}V(s,x),\partial_{xx}^{2}V(s,x),\bar{u}_{s})<0\text{
for }x\uparrow X_{s}^{t,x,\bar{u}}.
\]
Thus we obtain $\partial_{sx}^{2}V(s,X_{s}^{t,x,\bar{u}})+\frac{1}{2}%
\bar{\sigma}^{2}\partial_{x}F(s)\geq0$ and $\partial_{sx}^{2}V(s,X_{s}%
^{t,x,\bar{u}})+\frac{1}{2}\underline{\sigma}^{2}\partial_{x}F(s)\leq0$ by
(\ref{ee3-7}) and (\ref{ee3-8}), which implies $\partial_{sx}^{2}%
V(s,X_{s}^{t,x,\bar{u}})=-\frac{1}{2}v_{s}\partial_{x}F(s)$ for some $v_{s}%
\in \lbrack \underline{\sigma}^{2},\bar{\sigma}^{2}]$. By the same analysis, we
can get $\partial_{sx}^{2}V(s,X_{s}^{t,x,\bar{u}})=-\frac{1}{2}v_{s}%
\partial_{x}F(s)$ under the case $F(s)=0$ and $\partial_{x}F(s)<0$.

(3) Applying Ito's formula to $\partial_{x}V(s,X_{s}^{t,x,\bar{u}})$ under
$P_{t,x}^{\ast}$, we have%
\[
\left \{
\begin{array}
[c]{rl}%
d\partial_{x}V(s,X_{s}^{t,x,\bar{u}})= & [\partial_{xx}^{2}V(s,X_{s}%
^{t,x,\bar{u}})h(s,X_{s}^{t,x,\bar{u}},u_{s})+\frac{1}{2}\partial_{xxx}%
^{3}V(s,X_{s}^{t,x,\bar{u}})\sigma^{2}(s,X_{s}^{t,x,\bar{u}},u_{s})]d\langle
B\rangle_{s}\\
& +\partial_{sx}^{2}V(s,X_{s}^{t,x,\bar{u}})ds+\partial_{xx}^{2}%
V(s,X_{s}^{t,x,\bar{u}})\sigma(s,X_{s}^{t,x,\bar{u}},u_{s})dB_{s}\\
\partial_{x}V(T,X_{T}^{t,x,\bar{u}})= & \Phi^{\prime}(X_{T}^{t,x,\bar{u}%
}),\text{ }s\in \lbrack t,T].
\end{array}
\right.
\]
By (\ref{eee3-3}) and (\ref{eee3-4}), we know $\partial_{sx}^{2}%
V(s,X_{s}^{t,x,\bar{u}})ds=-\frac{1}{2}\partial_{x}F(s)d\langle B\rangle_{s}$
$P_{t,x}^{\ast}$-a.s. Then, by (\ref{ee3-1}), it is easy to verify that%
\[
(p_{s},q_{s},N_{s})=(\partial_{x}V(s,X_{s}^{t,x,\bar{u}}),\sigma
(s,X_{s}^{t,x,\bar{u}},\bar{u}_{s})\partial_{xx}^{2}V(s,X_{s}^{t,x,\bar{u}%
}),0)
\]
satisfies the adjoint equation (\ref{e2-4}) under $P_{t,x}^{\ast}$, which
implies (\ref{ee3-3}).

It follows from (\ref{ee3-2}) that, $P_{t,x}^{\ast}$-a.s., for a.e.
$s\in \lbrack t,T]$,
\[
G(F(s,X_{s}^{t,x,\bar{u}},V(s,X_{s}^{t,x,\bar{u}}),\partial_{x}V(s,X_{s}%
^{t,x,\bar{u}}),\partial_{xx}^{2}V(s,X_{s}^{t,x,\bar{u}}),u))\geq G(F(s)).
\]
Under the case $F(s)>0$, for any given $u\in U$, we know $u_{s}^{\varepsilon
}=\bar{u}_{s}+\varepsilon(u-\bar{u}_{s})\in U$ with $\varepsilon \in
\lbrack0,1]$ and%
\begin{align*}
&  G(F(s,X_{s}^{t,x,\bar{u}},V(s,X_{s}^{t,x,\bar{u}}),\partial_{x}%
V(s,X_{s}^{t,x,\bar{u}}),\partial_{xx}^{2}V(s,X_{s}^{t,x,\bar{u}}%
),u_{s}^{\varepsilon}))-G(F(s))\\
&  =\frac{1}{2}\bar{\sigma}^{2}[F(s,X_{s}^{t,x,\bar{u}},V(s,X_{s}^{t,x,\bar
{u}}),\partial_{x}V(s,X_{s}^{t,x,\bar{u}}),\partial_{xx}^{2}V(s,X_{s}%
^{t,x,\bar{u}}),u_{s}^{\varepsilon})-F(s)]\\
&  \geq0\text{ for }\varepsilon \downarrow0.
\end{align*}
From this, it is easy to deduce that
\begin{equation}
\langle \partial_{u}F(s,X_{s}^{t,x,\bar{u}},V(s,X_{s}^{t,x,\bar{u}}%
),\partial_{x}V(s,X_{s}^{t,x,\bar{u}}),\partial_{xx}^{2}V(s,X_{s}^{t,x,\bar
{u}}),\bar{u}_{s}),u-\bar{u}_{s}\rangle \geq0\text{, }\forall u\in U\text{.}
\label{ee3-9}%
\end{equation}
By the same method, it is easy to check that (\ref{ee3-9}) still holds under
the cases $F(s)<0$ and $F(s)=0$. Thus it is easy to get (\ref{ee3-4}) by
(\ref{ee3-9}).
\end{proof}

\begin{remark}
Although we obtain (\ref{eee3-3}), we do not know whether there exists a
$P_{t,x}^{\ast}\in \mathcal{\tilde{P}}_{t,x}^{\ast}$ such that (\ref{eee3-4})
holds (see Remark 2.3 in \cite{STZ11}).
\end{remark}

The following proposition gives a sufficient condition for the assumption
(\ref{eee3-4}) to hold.

\begin{proposition}
Suppose that Assumptions (H1) and (H2) hold. Let $\bar{u}(\cdot)\in
\mathcal{U}^{t}[t,T]$ be an optimal control for the control problem
(\ref{e2-3}) and let $(X^{t,x,\bar{u}},Y^{t,x,\bar{u}},Z^{t,x,\bar{u}%
},K^{t,x,\bar{u}})$ be the corresponding trajectory. The value function
$V(\cdot)\in C^{1,2}([t,T]\times \mathbb{R})$, $\partial_{xx}^{2}V(\cdot)$ is
of polynomial growth in $x$ and $\partial_{sx}^{2}V(\cdot)$ is continuous. If
$h(\cdot)$, $\sigma(\cdot)$ are bounded functions and there exists a $\beta>0$
such that $\sigma^{2}(\cdot)\geq \beta$, then the assumption (\ref{eee3-4})
holds for any $P_{t,x}^{\ast}\in \mathcal{\tilde{P}}_{t,x}^{\ast}$, where
$\mathcal{\tilde{P}}_{t,x}^{\ast}$ is defined in (\ref{eee3-1}).
\end{proposition}

\begin{proof}
We only need to show that, for a.e. $s\in \lbrack t,T]$,%
\[
\partial_{sx}^{2}V(s,X_{s}^{t,x,\bar{u}})=-\frac{1}{2}\gamma_{s}\partial
_{x}F(s)\text{ }P_{t,x}^{\ast}\text{-a.s.}%
\]
Set%
\[
A_{s}=\{x\in \mathbb{R}:\partial_{s}V(s,x)=0,\text{ }\partial_{sx}%
^{2}V(s,x)\not =0\}.
\]
For each $x\in A_{s}$, it is easy to find a $\delta>0$ such that $\partial
_{s}V(s,x^{\prime})\not =0$ for $x^{\prime}\in(x-\delta,x)\cup(x,x+\delta)$,
which implies that $A_{s}$ is countable. By (\ref{eee3-3}), it is easy to
deduce that%
\[
\left \{  \partial_{sx}^{2}V(s,X_{s}^{t,x,\bar{u}})\not =-\frac{1}{2}\gamma
_{s}\partial_{x}F(s)\right \}  \subseteq \left \{  X_{s}^{t,x,\bar{u}}\in
A_{s}\right \}  .
\]
By Theorem 3.7 in \cite{HWZ}, we know that $c(\left \{  X_{s}^{t,x,\bar{u}}\in
A_{s}\right \}  )=0$, which implies $P_{t,x}^{\ast}(\left \{  X_{s}^{t,x,\bar
{u}}\in A_{s}\right \}  )=0$. Thus we obtain (\ref{eee3-4}) for any
$P_{t,x}^{\ast}\in \mathcal{\tilde{P}}_{t,x}^{\ast}$.
\end{proof}

Now we study the relationship between MP and DPP in the viscosity sense. We
recall the definition of the first-order super-jet and sub-jet of $V(\cdot)\in
C([0,T]\times \mathbb{R})$ in $x$ (see \cite{CIP}) : for each fixed
$(t,x)\in \lbrack0,T]\times \mathbb{R}$,%
\[
\left \{
\begin{array}
[c]{rl}%
D_{x}^{1,+}V(t,x)= & \{a\in \mathbb{R}:V(t,x^{\prime})\leq V(t,x)+a(x^{\prime
}-x)+o(|x^{\prime}-x|)\text{ as }x^{\prime}\rightarrow x\},\\
D_{x}^{1,-}V(t,x)= & \{a\in \mathbb{R}:V(t,x^{\prime})\geq V(t,x)+a(x^{\prime
}-x)+o(|x^{\prime}-x|)\text{ as }x^{\prime}\rightarrow x\}.
\end{array}
\right.
\]
For simplicity, the constant $C>0$ will change from line to line in the
following, and we only consider $D_{x}^{1,+}V(t,x)$ and $D_{x}^{1,-}V(t,x)$.

\begin{theorem}
\label{MP-DPP-2}Suppose that Assumptions (H1) and (H2) hold. Let $\bar
{u}(\cdot)\in \mathcal{U}^{t}[t,T]$ be an optimal control for the control
problem (\ref{e2-3}) and let $(X^{t,x,\bar{u}},Y^{t,x,\bar{u}},Z^{t,x,\bar{u}%
},K^{t,x,\bar{u}})$ be the corresponding trajectory. Then we have%
\[
D_{x}^{1,-}V(t,x)\subseteq \lbrack \tilde{p}_{t},\bar{p}_{t}],
\]
where $\bar{p}_{t}=\sup_{P^{x}\in \mathcal{P}_{t,x}^{\ast}}p_{t}^{P^{x}}$,
$\tilde{p}_{t}=\inf_{P^{x}\in \mathcal{P}_{t,x}^{\ast}}p_{t}^{P^{x}}$,
$\mathcal{P}_{t,x}^{\ast}$ is defined in (\ref{e2-3-1}), $(p^{P^{x}}%
(\cdot),q^{P^{x}}(\cdot),N^{P^{x}}(\cdot))$ is the solution of the adjoint
equation (\ref{e2-4}) under $P^{x}$. Moreover, if $\bar{p}_{t}=\tilde{p}_{t}$,
then
\[
\{ \bar{p}_{t}\} \subseteq D_{x}^{1,+}V(t,x).
\]

\end{theorem}

\begin{proof}
For each $x^{\prime}\in \mathbb{R}$ and $s\in \lbrack t,T]$, set%
\[%
\begin{array}
[c]{ll}%
\hat{X}_{s}=X_{s}^{t,x^{\prime},\bar{u}}-X_{s}^{t,x,\bar{u}}, & \hat{Y}%
_{s}=Y_{s}^{t,x^{\prime},\bar{u}}-Y_{s}^{t,x,\bar{u}},\\
\hat{Z}_{s}=Z_{s}^{t,x^{\prime},\bar{u}}-Z_{s}^{t,x,\bar{u}}, & \hat{K}%
_{s}=K_{s}^{t,x^{\prime},\bar{u}}-K_{s}^{t,x,\bar{u}}.
\end{array}
\]
For simplicity, let $x$ be a fixed constant and $x^{\prime}\in \lbrack
x-1,x+1]$. The proof is divided into six steps.

Step 1: Estimates for $\hat{X}$, $\hat{Y}$, $\hat{Z}$ and $\hat{K}$ under
$\mathbb{\hat{E}}[\cdot]$.

By the estimates of $G$-SDEs (see \cite{HJPS, P2019}), we have, for each
$p\geq2$,%
\begin{equation}
\mathbb{\hat{E}}\left[  \underset{s\in \lbrack t,T]}{\sup}|\hat{X}_{s}%
|^{p}\right]  \leq C|x^{\prime}-x|^{p}\text{ and }\mathbb{\hat{E}}\left[
\underset{s\in \lbrack t,T]}{\sup}|X_{s}^{t,x,\bar{u}}|^{p}\right]  \leq C,
\label{ee3-10}%
\end{equation}
where the constant $C>0$ depends on $T$, $\bar{\sigma}^{2}$, $L$ and $p$. It
follows from Proposition 5.1 in \cite{HJPS1} that, for each $p\geq2$,
\[
|\hat{Y}_{s}|^{p}\leq C\mathbb{\hat{E}}_{s}\left[  \left \vert \Phi
(X_{T}^{t,x^{\prime},\bar{u}})-\Phi(X_{T}^{t,x,\bar{u}})\right \vert
^{p}+\left(  \int_{t}^{T}|\hat{g}_{r}|dr\right)  ^{p}\right]  ,
\]
where $C>0$ depends on $T$, $\bar{\sigma}^{2}$, $L$ and $p$, $\hat{g}%
_{r}=g(r,X_{r}^{t,x^{\prime},\bar{u}},Y_{r}^{t,x,\bar{u}},Z_{r}^{t,x,\bar{u}%
},\bar{u}_{r})-g(r,X_{r}^{t,x,\bar{u}},Y_{r}^{t,x,\bar{u}},Z_{r}^{t,x,\bar{u}%
},\bar{u}_{r})$. By Doob's inequality under $\mathbb{\hat{E}}[\cdot]$ (see
\cite{STZ, Song11}), we have%
\[
\mathbb{\hat{E}}\left[  \underset{s\in \lbrack t,T]}{\sup}|\hat{Y}_{s}%
|^{p}\right]  \leq C\left(  \mathbb{\hat{E}}\left[  \left \vert \Phi
(X_{T}^{t,x^{\prime},\bar{u}})-\Phi(X_{T}^{t,x,\bar{u}})\right \vert
^{p+1}+\left(  \int_{t}^{T}|\hat{g}_{r}|dr\right)  ^{p+1}\right]  \right)
^{p/(p+1)}.
\]
By (H2), (\ref{ee3-10}) and H\"{o}lder's inequality, we can easily obtain%
\begin{equation}
\mathbb{\hat{E}}\left[  \underset{s\in \lbrack t,T]}{\sup}|\hat{Y}_{s}%
|^{p}\right]  \leq C|x^{\prime}-x|^{p}, \label{ee3-11}%
\end{equation}
where $C>0$ depends on $T$, $\bar{\sigma}^{2}$, $L$ and $p$. By Proposition
3.8 in \cite{HJPS1}, we deduce%
\begin{equation}
\mathbb{\hat{E}}\left[  \left(  \int_{t}^{T}|\hat{Z}_{s}|^{2}ds\right)
^{p/2}+|\hat{K}_{T}|^{p}\right]  \leq C|x^{\prime}-x|^{p/2}\text{ for }p\geq2,
\label{ee3-12}%
\end{equation}
where $C>0$ depends on $T$, $\bar{\sigma}^{2}$, $\underline{\sigma}^{2}$, $L$
and $p$.

Step 2: Estimates for $\hat{Z}$ and $\hat{K}$ under $P\in \mathcal{P}$.

Set%
\[
\mathcal{\tilde{P}}_{t,x^{\prime}}^{\ast}=\{P\in \mathcal{P}:E_{P}%
[K_{T}^{t,x^{\prime},\bar{u}}]=0\}.
\]
For each fixed $P^{x^{\prime}}\in \mathcal{\tilde{P}}_{t,x^{\prime}}^{\ast}$,
applying Ito's formula to $|\hat{Y}_{s}|^{2}$ under $P^{x^{\prime}}$, we
obtain%
\[
|\hat{Y}_{t}|^{2}+\int_{t}^{T}|\hat{Z}_{s}|^{2}d\langle B\rangle_{s}=|\hat
{Y}_{T}|^{2}+2\int_{t}^{T}\hat{Y}_{s}\tilde{g}_{s}d\langle B\rangle_{s}%
-2\int_{t}^{T}\hat{Y}_{s}\hat{Z}_{s}dB_{s}+2\int_{t}^{T}\hat{Y}_{s}%
dK_{s}^{t,x,\bar{u}},
\]
where $\tilde{g}_{s}=g(s,X_{s}^{t,x^{\prime},\bar{u}},Y_{s}^{t,x^{\prime}%
,\bar{u}},Z_{s}^{t,x^{\prime},\bar{u}},\bar{u}_{s})-g(s,X_{s}^{t,x,\bar{u}%
},Y_{s}^{t,x,\bar{u}},Z_{s}^{t,x,\bar{u}},\bar{u}_{s})$. By the
Burkholder-Davis-Gundy inequality, (\ref{ee3-10}) and (\ref{ee3-11}), we
obtain%
\begin{equation}
E_{P^{x^{\prime}}}\left[  \left(  \int_{t}^{T}|\hat{Z}_{s}|^{2}ds\right)
^{p/2}\right]  \leq C|x-x^{\prime}|^{p}+CE_{P^{x^{\prime}}}\left[
|K_{T}^{t,x,\bar{u}}|^{p/2}\underset{s\in \lbrack t,T]}{\sup}|\hat{Y}%
_{s}|^{p/2}\right]  , \label{ee3-13}%
\end{equation}
where $C>0$ depends on $T$, $\bar{\sigma}^{2}$, $\underline{\sigma}^{2}$, $L$
and $p$. Since%
\[
-K_{T}^{t,x,\bar{u}}=\hat{Y}_{T}-\hat{Y}_{t}+\int_{t}^{T}\tilde{g}_{s}d\langle
B\rangle_{s}-\int_{t}^{T}\hat{Z}_{s}dB_{s}\text{ }P^{x^{\prime}}\text{-a.s.,}%
\]
we get%
\begin{equation}
E_{P^{x^{\prime}}}\left[  |K_{T}^{t,x,\bar{u}}|^{p}\right]  \leq
C|x-x^{\prime}|^{p}+CE_{P^{x^{\prime}}}\left[  \left(  \int_{t}^{T}|\hat
{Z}_{s}|^{2}ds\right)  ^{p/2}\right]  , \label{ee3-14}%
\end{equation}
where $C>0$ depends on $T$, $\bar{\sigma}^{2}$, $\underline{\sigma}^{2}$, $L$
and $p$. Combining (\ref{ee3-13}) and (\ref{ee3-14}), we get%
\begin{equation}
E_{P^{x^{\prime}}}\left[  \left(  \int_{t}^{T}|\hat{Z}_{s}|^{2}ds\right)
^{p/2}+|K_{T}^{t,x,\bar{u}}|^{p}\right]  \leq C|x^{\prime}-x|^{p},
\label{ee3-15}%
\end{equation}
where $C>0$ depends on $T$, $\bar{\sigma}^{2}$, $\underline{\sigma}^{2}$, $L$
and $p$. By similar method, we can obtain that, for each fixed $P^{x}%
\in \mathcal{P}_{t,x}^{\ast}$ and $p\geq2$,%
\begin{equation}
E_{P^{x}}\left[  \left(  \int_{t}^{T}|\hat{Z}_{s}|^{2}ds\right)  ^{p/2}%
+|K_{T}^{t,x^{\prime},\bar{u}}|^{p}\right]  \leq C|x^{\prime}-x|^{p},
\label{ee3-16}%
\end{equation}
where $C>0$ depends on $T$, $\bar{\sigma}^{2}$, $\underline{\sigma}^{2}$, $L$
and $p$.

Step 3: Variation of $\hat{X}$ and $\hat{Y}$.

Rewrite the equation of $(\hat{X},\hat{Y},\hat{Z},\hat{K})$ as follows:%
\begin{equation}
\left \{
\begin{array}
[c]{rl}%
d\hat{X}_{s}= & [h_{x}(s)+\varepsilon_{1}(s)]\hat{X}_{s}d\langle B\rangle
_{s}+[\sigma_{x}(s)+\varepsilon_{2}(s)]\hat{X}_{s}dB_{s},\\
d\hat{Y}_{s}= & -\{[g_{x}(s)+\varepsilon_{3}(s)]\hat{X}_{s}+[g_{y}%
(s)+\varepsilon_{4}(s)]\hat{Y}_{s}+[g_{z}(s)+\varepsilon_{5}(s)]\hat{Z}%
_{s}\}d\langle B\rangle_{s}\\
& +\hat{Z}_{s}dB_{s}+dK_{s}^{t,x^{\prime},\bar{u}}-dK_{s}^{t,x,\bar{u}},\\
\hat{X}_{t}= & x^{\prime}-x,\text{ }\hat{Y}_{T}=[\Phi^{\prime}(X_{T}%
^{t,x,\bar{u}})+\varepsilon_{6}(T)]\hat{X}_{T},\text{ }s\in \lbrack t,T],
\end{array}
\right.  \label{ee3-17}%
\end{equation}
where%
\[%
\begin{array}
[c]{rl}%
\varepsilon_{1}(s)= & \int_{0}^{1}[h_{x}(s,X_{s}^{t,x,\bar{u}}+\alpha \hat
{X}_{s},\bar{u}_{s})-h_{x}(s)]d\alpha,\\
\varepsilon_{2}(s)= & \int_{0}^{1}[\sigma_{x}(s,X_{s}^{t,x,\bar{u}}+\alpha
\hat{X}_{s},\bar{u}_{s})-\sigma_{x}(s)]d\alpha,\\
\varepsilon_{3}(s)= & \int_{0}^{1}[g_{x}(s,X_{s}^{t,x,\bar{u}}+\alpha \hat
{X}_{s},Y_{s}^{t,x,\bar{u}}+\alpha \hat{Y}_{s},Z_{s}^{t,x,\bar{u}}+\alpha
\hat{Z}_{s},\bar{u}_{s})-g_{x}(s)]d\alpha,\\
\varepsilon_{4}(s)= & \int_{0}^{1}[g_{y}(s,X_{s}^{t,x,\bar{u}}+\alpha \hat
{X}_{s},Y_{s}^{t,x,\bar{u}}+\alpha \hat{Y}_{s},Z_{s}^{t,x,\bar{u}}+\alpha
\hat{Z}_{s},\bar{u}_{s})-g_{y}(s)]d\alpha,\\
\varepsilon_{5}(s)= & \int_{0}^{1}[g_{z}(s,X_{s}^{t,x,\bar{u}}+\alpha \hat
{X}_{s},Y_{s}^{t,x,\bar{u}}+\alpha \hat{Y}_{s},Z_{s}^{t,x,\bar{u}}+\alpha
\hat{Z}_{s},\bar{u}_{s})-g_{z}(s)]d\alpha,\\
\varepsilon_{6}(T)= & \int_{0}^{1}[\Phi^{\prime}(X_{T}^{t,x,\bar{u}}%
+\alpha \hat{X}_{T})-\Phi^{\prime}(X_{T}^{t,x,\bar{u}})]d\alpha.
\end{array}
\]
Let $(l_{s})_{s\in \lbrack t,T]}$ be the solution of the following $G$-SDE:%
\[
dl_{s}=g_{y}(s)l_{s}d\langle B\rangle_{s}+g_{z}(s)l_{s}dB_{s},\text{ }%
l_{t}=1.
\]
For each given $P^{x^{\prime}}\in \mathcal{\tilde{P}}_{t,x^{\prime}}^{\ast}$,
the solution of the adjoint equation (\ref{e2-4}) is denoted by
$(p^{P^{x^{\prime}}}(\cdot),q^{P^{x^{\prime}}}(\cdot),N^{P^{x^{\prime}}}%
(\cdot))$. Applying Ito's formula to $l_{s}(\hat{Y}_{s}-p_{s}^{P^{x^{\prime}}%
}\hat{X}_{s})$ under $P^{x^{\prime}}$ and noting that $dK_{s}^{t,x^{\prime
},\bar{u}}=0$, we obtain%
\begin{equation}%
\begin{array}
[c]{rl}%
\hat{Y}_{t}-p_{t}^{P^{x^{\prime}}}\hat{X}_{t}= & E_{P^{x^{\prime}}}\left[
\varepsilon_{6}(T)l_{T}\hat{X}_{T}+\int_{t}^{T}l_{s}dK_{s}^{t,x,\bar{u}}%
+\int_{t}^{T}\left \{  \varepsilon_{4}(s)l_{s}\hat{Y}_{s}+\varepsilon
_{5}(s)l_{s}\hat{Z}_{s}\right.  \right. \\
& \left.  \left.  +\left[  \varepsilon_{1}(s)p_{s}^{P^{x^{\prime}}%
}+\varepsilon_{2}(s)q_{s}^{P^{x^{\prime}}}+\varepsilon_{2}(s)g_{z}%
(s)p_{s}^{P^{x^{\prime}}}+\varepsilon_{3}(s)\right]  l_{s}\hat{X}_{s}\right \}
d\langle B\rangle_{s}\right]  .
\end{array}
\label{ee3-18}%
\end{equation}

Step 4: Estimates of every terms in the right side of (\ref{ee3-18}).

By $l_{s}\geq0$ and $dK_{s}^{t,x,\bar{u}}\leq0$, we get $\int_{t}^{T}%
l_{s}dK_{s}^{t,x,\bar{u}}\leq0$. It follows from H\"{o}lder's inequality,
$d\langle B\rangle_{s}\leq \bar{\sigma}^{2}ds$ and (\ref{ee3-15}) that
\begin{align*}
&  \left \vert E_{P^{x^{\prime}}}\left[  \int_{t}^{T}\varepsilon_{5}%
(s)l_{s}\hat{Z}_{s}d\langle B\rangle_{s}\right]  \right \vert \\
&  \leq C\left(  E_{P^{x^{\prime}}}\left[  \int_{t}^{T}|\hat{Z}_{s}%
|^{2}ds\right]  \right)  ^{1/2}\left(  E_{P^{x^{\prime}}}\left[  \int_{t}%
^{T}|\varepsilon_{5}(s)|^{4}ds\right]  \right)  ^{1/4}\left(  E_{P^{x^{\prime
}}}\left[  \int_{t}^{T}|l_{s}|^{4}ds\right]  \right)  ^{1/4}\\
&  \leq C\left(  \mathbb{\hat{E}}\left[  \int_{t}^{T}|\varepsilon_{5}%
(s)|^{4}ds\right]  \right)  ^{1/4}\left(  \mathbb{\hat{E}}\left[
\underset{s\in \lbrack t,T]}{\sup}|l_{s}|^{4}\right]  \right)  ^{1/4}%
|x^{\prime}-x|,
\end{align*}
where $C>0$ depends on $T$, $\bar{\sigma}^{2}$, $\underline{\sigma}^{2}$ and
$L$. Similar to (\ref{ee3-10}), we know
\[
\mathbb{\hat{E}}\left[  \underset{s\in \lbrack t,T]}{\sup}|l_{s}|^{4}\right]
\leq C,
\]
where $C>0$ depends on $T$, $\bar{\sigma}^{2}$ and $L$. Now, we prove%
\begin{equation}
\mathbb{\hat{E}}\left[  \int_{t}^{T}|\varepsilon_{5}(s)|^{4}ds\right]
\rightarrow0\text{ as }x^{\prime}\rightarrow x. \label{ee3-19}%
\end{equation}
For each given $N>0$ and $\delta \in \lbrack0,1]$, set%
\[%
\begin{array}
[c]{rl}%
\bar{\omega}_{N}(\delta)= & \sup \left \{  g_{z}(s,x+x^{\prime},y+y^{\prime
},z+z^{\prime},u)-g_{z}(s,x,y,z,u):|x|\leq N,\right. \\
& \left.  |y|\leq N,\text{ }|z|\leq N,\text{ }u\in U,\text{ }|x^{\prime}%
|\leq \delta,\text{ }|y^{\prime}|\leq \delta,\text{ }|z^{\prime}|\leq
\delta \right \}  .
\end{array}
\]
By (H1), we have $\bar{\omega}_{N}(\delta)\rightarrow0$ as $\delta
\rightarrow0$ for each fixed $N>0$. Thus%
\[%
\begin{array}
[c]{rl}%
|\varepsilon_{5}(s)|\leq & \bar{\omega}_{N}(\delta)+\frac{2L}{\sqrt{\delta}%
}\left(  \sqrt{|\hat{X}_{s}|}+\sqrt{|\hat{Y}_{s}|}+\sqrt{|\hat{Z}_{s}|}\right)
\\
& +\frac{2L}{\sqrt{N}}\left(  \sqrt{|X_{s}^{t,x,\bar{u}}|}+\sqrt
{|Y_{s}^{t,x,\bar{u}}|}+\sqrt{|Z_{s}^{t,x,\bar{u}}|}\right)  .
\end{array}
\]
By (\ref{ee3-10}), (\ref{ee3-11}) and (\ref{ee3-12}), we obtain%
\[
\mathbb{\hat{E}}\left[  \int_{t}^{T}|\varepsilon_{5}(s)|^{4}ds\right]  \leq
C\left(  |\bar{\omega}_{N}(\delta)|^{4}+\frac{1}{\delta^{2}}|x^{\prime
}-x|+\frac{1}{N^{2}}\right)  ,
\]
where $C>0$ depends on $T$, $\bar{\sigma}^{2}$, $\underline{\sigma}^{2}$ and
$L$. Thus we get%
\[
\underset{x^{\prime}\rightarrow x}{\lim \sup}\mathbb{\hat{E}}\left[  \int
_{t}^{T}|\varepsilon_{5}(s)|^{4}ds\right]  \leq C\left(  |\bar{\omega}%
_{N}(\delta)|^{4}+\frac{1}{N^{2}}\right)  ,
\]
which implies (\ref{ee3-19}) by letting $\delta \rightarrow0$ and then
$N\rightarrow \infty$. Then we obtain%
\begin{equation}
E_{P^{x^{\prime}}}\left[  \int_{t}^{T}\varepsilon_{5}(s)l_{s}\hat{Z}%
_{s}d\langle B\rangle_{s}\right]  =o(|x^{\prime}-x|). \label{ee3-20}%
\end{equation}
Similar to the proof of (\ref{ee3-20}), we can prove that the other terms in
the right side of (\ref{ee3-18}) are $o(|x^{\prime}-x|)$. Note that $\hat
{X}_{t}=x^{\prime}-x$, then we get%
\begin{equation}
\hat{Y}_{t}\leq p_{t}^{P^{x^{\prime}}}(x^{\prime}-x)+o(|x^{\prime}-x|)\text{
for }P^{x^{\prime}}\in \mathcal{\tilde{P}}_{t,x^{\prime}}^{\ast}.
\label{ee3-21}%
\end{equation}
Similar to the proof of (\ref{ee3-21}), we can deduce%
\begin{equation}
\hat{Y}_{t}\geq p_{t}^{P^{x}}(x^{\prime}-x)+o(|x^{\prime}-x|)\text{ for }%
P^{x}\in \mathcal{P}_{t,x}^{\ast}. \label{ee3-22}%
\end{equation}

Step 5: First order derivative for $\hat{Y}_{t}$.

For $x^{\prime}>x$, we get by (\ref{ee3-22}) that%
\begin{equation}
\underset{x^{\prime}\downarrow x}{\lim \inf}\frac{\hat{Y}_{t}}{x^{\prime}%
-x}=\underset{x^{\prime}\downarrow x}{\lim \inf}\frac{Y_{t}^{t,x^{\prime}%
,\bar{u}}-Y_{t}^{t,x,\bar{u}}}{x^{\prime}-x}\geq \sup_{P^{x}\in \mathcal{P}%
_{t,x}^{\ast}}p_{t}^{P^{x}}=\bar{p}_{t}. \label{ee3-23}%
\end{equation}
It follows from (\ref{ee3-11}) that $\hat{Y}_{t}(x^{\prime}-x)^{-1}$ is
bounded. Thus we can choose $x^{n}>x$, $n\geq1$, such that $x^{n}\downarrow x$
and%
\begin{equation}
\underset{x^{\prime}\downarrow x}{\lim \sup}\frac{Y_{t}^{t,x^{\prime},\bar{u}%
}-Y_{t}^{t,x,\bar{u}}}{x^{\prime}-x}=\lim_{n\rightarrow \infty}\frac
{Y_{t}^{t,x^{n},\bar{u}}-Y_{t}^{t,x,\bar{u}}}{x^{n}-x}<\infty. \label{ee3-26}%
\end{equation}
Since $P^{x^{n}}\in \mathcal{\tilde{P}}_{t,x^{n}}^{\ast}\subset \mathcal{P}$ and
$\mathcal{P}$\ is weakly compact, there exists a subsequence of $\{x^{n}%
:n\geq1\}$, denoted by $\{x_{i}^{n}:i\geq1\}$, such that $P^{x_{i}^{n}%
}\rightarrow P^{\ast}\in \mathcal{P}$ weakly as $i\rightarrow \infty$. By Lemma
29 in \cite{DHP11}, we have%
\begin{equation}
E_{P^{x_{i}^{n}}}\left[  K_{T}^{t,x,\bar{u}}\right]  \rightarrow E_{P^{\ast}%
}\left[  K_{T}^{t,x,\bar{u}}\right]  \text{ as }i\rightarrow \infty.
\label{ee3-24}%
\end{equation}
Noting that $K_{T}^{t,x_{i}^{n},\bar{u}}=0$ $P^{x_{i}^{n}}$-a.s., we get
\begin{equation}
-E_{P^{x_{i}^{n}}}\left[  K_{T}^{t,x,\bar{u}}\right]  =E_{P^{x_{i}^{n}}%
}\left[  |K_{T}^{t,x_{i}^{n},\bar{u}}-K_{T}^{t,x,\bar{u}}|\right]
\leq \mathbb{\hat{E}}\left[  |K_{T}^{t,x_{i}^{n},\bar{u}}-K_{T}^{t,x,\bar{u}%
}|\right]  . \label{ee3-25}%
\end{equation}
By (\ref{ee3-12}), (\ref{ee3-24}) and (\ref{ee3-25}), we deduce $E_{P^{\ast}%
}\left[  K_{T}^{t,x,\bar{u}}\right]  =0$, which implies $P^{\ast}%
\in \mathcal{P}_{t,x}^{\ast}$. Since the adjoint equation (\ref{e2-4}) is a
linear BSDE, by Proposition 2.2 in \cite{EPQ}, we obtain%
\[
p_{t}^{P^{x_{i}^{n}}}=E_{P^{x_{i}^{n}}}\left[  \lambda_{T}\Phi^{\prime}%
(X_{T}^{t,x,\bar{u}})+\int_{t}^{T}\lambda_{s}g_{x}(s)d\langle B\rangle
_{s}\right]  ,
\]
where%
\[
\lambda_{s}=\exp \left(  \int_{t}^{s}\beta_{r}dB_{r}+\int_{t}^{s}(\alpha
_{r}-\frac{1}{2}|\beta_{r}|^{2})d\langle B\rangle_{r}\right)  ,
\]%
\[
\alpha_{r}=h_{x}(r)+g_{y}(r)+g_{z}(r)\sigma_{x}(r),\text{ }\beta_{r}%
=g_{z}(r)+\sigma_{x}(r).
\]
Since $\lambda_{T}\Phi^{\prime}(X_{T}^{t,x,\bar{u}})+\int_{t}^{T}\lambda
_{s}g_{x}(s)d\langle B\rangle_{s}\in L_{G}^{1}(\Omega_{T}^{t})$, we get
\begin{equation}
p_{t}^{P^{x_{i}^{n}}}\rightarrow p_{t}^{P^{\ast}}\text{ as }i\rightarrow
\infty \label{ee3-27}%
\end{equation}
by Lemma 29 in \cite{DHP11}. By (\ref{ee3-23}), (\ref{ee3-26}) and
(\ref{ee3-27}), we obtain%
\[
\lim_{x^{\prime}\downarrow x}\frac{Y_{t}^{t,x^{\prime},\bar{u}}-Y_{t}%
^{t,x,\bar{u}}}{x^{\prime}-x}=\bar{p}_{t}=\sup_{P^{x}\in \mathcal{P}%
_{t,x}^{\ast}}p_{t}^{P^{x}},
\]
which implies%
\begin{equation}
Y_{t}^{t,x^{\prime},\bar{u}}-Y_{t}^{t,x,\bar{u}}=\bar{p}_{t}(x^{\prime
}-x)+o(|x^{\prime}-x|)\text{ for }x^{\prime}>x\text{.} \label{ee3-28}%
\end{equation}
Similar to (\ref{ee3-28}), we can obtain%
\begin{equation}
Y_{t}^{t,x^{\prime},\bar{u}}-Y_{t}^{t,x,\bar{u}}=\tilde{p}_{t}(x^{\prime
}-x)+o(|x^{\prime}-x|)\text{ for }x^{\prime}<x\text{,} \label{ee3-29}%
\end{equation}
where $\tilde{p}_{t}=\inf_{P^{x}\in \mathcal{P}_{t,x}^{\ast}}p_{t}^{P^{x}}$.

Step 6: $D_{x}^{1,-}V(t,x)$ and $D_{x}^{1,+}V(t,x)$.

Noting that $V(t,x)=Y_{t}^{t,x,\bar{u}}$ and $V(t,x^{\prime})\leq
Y_{t}^{t,x^{\prime},\bar{u}}$, we have%
\begin{equation}
V(t,x^{\prime})-V(t,x)\leq Y_{t}^{t,x^{\prime},\bar{u}}-Y_{t}^{t,x,\bar{u}}.
\label{ee3-30}%
\end{equation}
For any given $a\in D_{x}^{1,-}V(t,x)$, by definition of $D_{x}^{1,-}V(t,x)$,
we get%
\begin{equation}
a(x^{\prime}-x)+o(|x^{\prime}-x|)\leq V(t,x^{\prime})-V(t,x)\text{.}
\label{ee3-31}%
\end{equation}
By (\ref{ee3-28})-(\ref{ee3-31}), we deduce $a\leq \bar{p}_{t}$ if $x^{\prime
}>x$ and $a\geq \tilde{p}_{t}$ if $x^{\prime}<x$. Thus $a\in \lbrack \tilde
{p}_{t},\bar{p}_{t}]$, which implies $D_{x}^{1,-}V(t,x)\subseteq \lbrack
\tilde{p}_{t},\bar{p}_{t}]$. If $\tilde{p}_{t}=\bar{p}_{t}$, by (\ref{ee3-28}%
)-(\ref{ee3-30}) and the definition of $D_{x}^{1,+}V(t,x)$, we have $\bar
{p}_{t}\in D_{x}^{1,+}V(t,x)$.
\end{proof}

The first example shows that $D_{x}^{1,+}V(t,x)$ may be empty if $\bar{p}%
_{t}\not =\tilde{p}_{t}$.

\begin{example}
\label{exa1}Consider the following simple control system:%
\[
\left \{
\begin{array}
[c]{l}%
dX_{s}^{t,x,u}=u_{s}d\langle B\rangle_{s},\text{ }X_{t}^{t,x,u}=x,\text{ }%
s\in \lbrack t,T],\\
dY_{s}^{t,x,u}=Z_{s}^{t,x,u}dB_{s}+dK_{s}^{t,x,u},\text{ }Y_{T}^{t,x,u}%
=(X_{T}^{t,x,u})^{2},
\end{array}
\right.
\]
where $U=\{1\}$, $\bar{\sigma}^{2}=1$ and $\underline{\sigma}^{2}=0.2$.
Obviously, $\bar{u}=1$ and the value function%
\[
V(t,x)=\mathbb{\hat{E}}\left[  (x+\langle B\rangle_{T}-\langle B\rangle
_{t})^{2}\right]  ,\text{ }(t,x)\in \lbrack0,T]\times \mathbb{R}.
\]
By Theorem 3.5.4 in \cite{P2019}, we have%
\begin{equation}
\mathbb{\hat{E}}\left[  (x+\langle B\rangle_{T}-\langle B\rangle_{t}%
)^{2}\right]  =\sup_{\underline{\sigma}^{2}\leq v\leq \bar{\sigma}^{2}%
}(x+v(T-t))^{2}. \label{ee3-32}%
\end{equation}
Thus we obtain%
\[
V(t,x)=\left \{
\begin{array}
[c]{ll}%
(x+(T-t))^{2}, & x\geq-0.6(T-t),\\
(x+0.2(T-t))^{2}, & x<-0.6(T-t).
\end{array}
\right.
\]
On the point $(t^{\ast},x^{\ast})=(t^{\ast},-0.6(T-t^{\ast}))$ for some
$t^{\ast}<T$, it is easy to verify that the maximum value in (\ref{ee3-32}) is
obtained at $v=\bar{\sigma}^{2}$ or $v=\underline{\sigma}^{2}$. Thus
$\{P^{\upsilon}:v=\bar{\sigma}$ or $v=\underline{\sigma}\} \subseteq
\mathcal{P}_{t^{\ast},x^{\ast}}^{\ast}$, where $P^{\upsilon}$ is a probability
measure on $(\Omega_{T},\mathcal{B}(\Omega_{T}))$ such that $\langle
B\rangle_{s}=v^{2}s$ $P^{v}$-a.s. Under this case,%
\begin{align*}
p_{t^{\ast}}^{P^{\bar{\sigma}}}  &  =2E_{P^{\bar{\sigma}}}[x^{\ast}+\langle
B\rangle_{T}-\langle B\rangle_{t^{\ast}}]=0.8(T-t^{\ast}),\\
p_{t^{\ast}}^{P^{\underline{\sigma}}}  &  =2E_{P^{\underline{\sigma}}}%
[x^{\ast}+\langle B\rangle_{T}-\langle B\rangle_{t^{\ast}}]=-0.8(T-t^{\ast}),
\end{align*}
which implies $\bar{p}_{t^{\ast}}\not =\tilde{p}_{t^{\ast}}$. It is easy to
check that%
\[
\lim_{x\downarrow x^{\ast}}\frac{V(t^{\ast},x)-V(t^{\ast},x^{\ast})}%
{x-x^{\ast}}=0.8(T-t^{\ast}),\text{ }\lim_{x\uparrow x^{\ast}}\frac{V(t^{\ast
},x)-V(t^{\ast},x^{\ast})}{x-x^{\ast}}=-0.8(T-t^{\ast}),
\]
which implies $D_{x}^{1,+}V(t^{\ast},x^{\ast})=\emptyset$ by the definition of
$D_{x}^{1,+}V(t^{\ast},x^{\ast})$.
\end{example}

The second example is a linear quadratic optimal control problem under
$\mathbb{\hat{E}}[\cdot]$ studied in \cite{HJ0}.

\begin{example}
Consider the following control system:%
\[
\left \{
\begin{array}
[c]{rl}%
dX_{s}^{t,x,u}= & (4X_{s}^{t,x,u}+u_{s})d\langle B\rangle_{s}+(X_{s}%
^{t,x,u}+u_{s})dB_{s},\\
dY_{s}^{t,x,u}= & -\frac{1}{2}[(X_{s}^{t,x,u})^{2}+(u_{s})^{2}]d\langle
B\rangle_{s}+Z_{s}^{t,x,u}dB_{s}+dK_{s}^{t,x,u},\\
X_{t}^{t,x,u}= & x,\text{ }Y_{T}^{t,x,u}=\frac{1}{2}(X_{T}^{t,x,u})^{2},\text{
}s\in \lbrack t,T],
\end{array}
\right.
\]
where $x\not =0$, $U=\mathbb{R}$, $\bar{\sigma}^{2}=1$ and $\underline{\sigma
}^{2}=0.2$. Let $\bar{u}(\cdot)\in \mathcal{U}^{t}[t,T]$ be an optimal control.
By Theorem \ref{MP}, there exists a $P_{t,x}^{\ast}\in \mathcal{P}_{t,x}^{\ast
}$ such that $E_{P_{t,x}^{\ast}}\left[  K_{T}^{t,x,\bar{u}}\right]  =0$ and%
\begin{equation}
p_{s}+q_{s}+\bar{u}_{s}=0,\text{ a.e. }s\in \lbrack t,T],\text{ }P_{t,x}^{\ast
}\text{-a.s.,} \label{ee3-33}%
\end{equation}
where $(p(\cdot),q(\cdot),N(\cdot))$ satisfies the following BSDE under
$P_{t,x}^{\ast}$:%
\[
\left \{
\begin{array}
[c]{rl}%
dp_{s}= & -[4p_{s}+q_{s}+X_{s}^{t,x,\bar{u}}]d\langle B\rangle_{s}+q_{s}%
dB_{s}+dN_{s},\\
p_{T}= & X_{T}^{t,x,\bar{u}},\text{ }s\in \lbrack t,T].
\end{array}
\right.
\]
Suppose that $p_{s}=P_{s}X_{s}^{t,x,\bar{u}}$, $P_{t,x}^{\ast}$-a.s. and%
\[
dP_{s}=\Lambda_{s}d\langle B\rangle_{s},\text{ }P_{T}=1,\text{ }P_{t,x}^{\ast
}\text{-a.s.}%
\]
Applying Ito's formula to $P_{s}X_{s}^{t,x,\bar{u}}$ under $P_{t,x}^{\ast}$,
we get
\begin{equation}
\left \{
\begin{array}
[c]{rl}%
\Lambda_{s}X_{s}^{t,x,\bar{u}}= & -9P_{s}X_{s}^{t,x,\bar{u}}-X_{s}%
^{t,x,\bar{u}}-2P_{s}\bar{u}_{s},\\
q_{s}= & P_{s}X_{s}^{t,x,\bar{u}}+P_{s}\bar{u}_{s},\text{ }N_{s}=0.
\end{array}
\right.  \label{ee3-34}%
\end{equation}
By (\ref{ee3-33}) and (\ref{ee3-34}), we obtain%
\[
\bar{u}_{s}=-\frac{2P_{s}}{1+P_{s}}X_{s}^{t,x,\bar{u}},\text{ }\Lambda
_{s}=-\frac{5P_{s}^{2}+10P_{s}+1}{1+P_{s}}.
\]
Suppose that $d\langle B\rangle_{s}=\bar{\sigma}^{2}ds=ds$ under
$P_{t,x}^{\ast}$. Then we obtain%
\begin{equation}
dP_{s}=-\frac{5P_{s}^{2}+10P_{s}+1}{1+P_{s}}ds,\text{ }P_{T}=1, \label{ee3-35}%
\end{equation}
which implies%
\[
d(1+P_{s})^{2}=[-10(1+P_{s})^{2}+8]ds,\text{ }P_{T}=1.
\]
The solution of (\ref{ee3-35}) is%
\begin{equation}
P_{s}=\sqrt{\frac{16}{5}e^{10(T-s)}+\frac{4}{5}}-1\text{, }s\in \lbrack0,T].
\label{ee3-36}%
\end{equation}

In the following, we prove that $d\langle B\rangle_{s}=ds$ under
$P_{t,x}^{\ast}$. Let $(P_{s})_{s\in \lbrack0,T]}$ be the solution of
(\ref{ee3-35}) and
\begin{equation}
\bar{u}_{s}=-\frac{2P_{s}}{1+P_{s}}X_{s}^{t,x,\bar{u}}. \label{ee3-37}%
\end{equation}
Applying Ito's formula to $\frac{1}{2}P_{s}(X_{s}^{t,x,\bar{u}})^{2}$, it is
easy to verify that%
\[%
\begin{array}
[c]{l}%
Y_{s}^{t,x,\bar{u}}=\frac{1}{2}P_{s}(X_{s}^{t,x,\bar{u}})^{2},\text{ }%
Z_{s}^{t,x,\bar{u}}=P_{s}(X_{s}^{t,x,\bar{u}})^{2}+P_{s}X_{s}^{t,x,\bar{u}%
}\bar{u}_{s},\\
K_{s}^{t,x,\bar{u}}=\int_{t}^{s}\frac{5P_{r}^{2}+10P_{r}+1}{2(1+P_{r})}%
(X_{r}^{t,x,\bar{u}})^{2}[d\langle B\rangle_{r}-dr].
\end{array}
\]
Since $P_{r}>0$ and $(X_{r}^{t,x,\bar{u}})^{2}>0$, we obtain that
$\mathcal{P}_{t,x}^{\ast}$ contains only one element $P_{t,x}^{\ast}$ such
that $d\langle B\rangle_{r}=dr$. By Theorem 5.4 in \cite{HJ0}, $\bar{u}$
defined in (\ref{ee3-37}) is an optimal control.

By Theorem \ref{th-DPP}, the value function $V(\cdot)$ satisfies the following
HJB equation:%
\[
\left \{
\begin{array}
[c]{l}%
\partial_{t}V(t,x)+\underset{v\in \mathbb{R}}{\inf}G((x+v)^{2}\partial_{xx}%
^{2}V(t,x)+2(4x+v)\partial_{x}V(t,x)+x^{2}+v^{2})=0,\\
V(T,x)=\frac{1}{2}x^{2}.
\end{array}
\right.
\]
It is easy to verify that $V(t,x)=\frac{1}{2}P_{t}x^{2}$ is the unique
solution of the above HJB equation, where $(P_{s})_{s\in \lbrack0,T]}$ is the
solution of (\ref{ee3-35}). Thus one can easily check that the relations
(\ref{ee3-1}) and (\ref{ee3-3}) in Theorem \ref{MP-DPP-1} hold.
\end{example}

The third example shows that the relation (\ref{ee3-1}) in Theorem
\ref{MP-DPP-1} holds only for $P_{t,x}^{\ast}\in \mathcal{\tilde{P}}%
_{t,x}^{\ast}$ defined in (\ref{eee3-1}).

\begin{example}
\label{exa3}Consider the following control system:%
\[
\left \{
\begin{array}
[c]{rl}%
dX_{s}^{t,x,u}= & u_{s}d\langle B\rangle_{s},\\
dY_{s}^{t,x,u}= & Z_{s}^{t,x,u}dB_{s}+dK_{s}^{t,x,u},\\
X_{t}^{t,x,u}= & x,\text{ }Y_{T}^{t,x,u}=-(X_{T}^{t,x,u}-1)^{2},\text{ }%
s\in \lbrack t,T],
\end{array}
\right.
\]
where $T>1$, $U=[1,2]$, $\bar{\sigma}^{2}=1$ and $\underline{\sigma}^{2}=0.5$.
Under this case, the value function%
\begin{equation}
V(t,x)=\inf_{u\in \mathcal{U}^{t}[t,T]}Y_{t}^{t,x,u}=\inf_{u\in \mathcal{U}%
^{t}[t,T]}\mathbb{\hat{E}}\left[  -\left(  x+\int_{t}^{T}u_{s}d\langle
B\rangle_{s}-1\right)  ^{2}\right]  . \label{ee3-38}%
\end{equation}
By Theorem \ref{th-DPP}, the value function $V(\cdot)$ satisfies the following
HJB equation:%
\[
\partial_{t}V(t,x)+\underset{v\in U}{\inf}G(2v\partial_{x}V(t,x))=0,\text{
}V(T,x)=-(x-1)^{2}.
\]
It is easy to check that
\begin{equation}
V(t,x)=-(x+T-t-1)^{2} \label{ee3-39}%
\end{equation}
satisfies the above HJB equation.

Let $(t,x)=(T-1,0)$ be fixed. From (\ref{ee3-38}) and (\ref{ee3-39}), we
obtain that $\bar{u}_{s}=v$, $s\in \lbrack t,T]$, is an optimal control for any
fixed $v\in U$. Take $\bar{u}_{s}=v$ with $v\in(1,2)$, it is easy to check
that%
\[%
\begin{array}
[c]{l}%
Y_{s}^{t,x,\bar{u}}=\phi(s,v(\langle B\rangle_{s}-\langle B\rangle
_{t})-1),\text{ }Z_{s}^{t,x,\bar{u}}=0,\\
K_{s}^{t,x,\bar{u}}=\int_{t}^{s}\partial_{x}\phi(r,v(\langle B\rangle
_{r}-\langle B\rangle_{t})-1)vd\langle B\rangle_{r}+\int_{t}^{s}\partial
_{r}\phi(r,v(\langle B\rangle_{r}-\langle B\rangle_{t})-1)dr,
\end{array}
\]
where%
\[
\phi(s,x^{\prime})=\left \{
\begin{array}
[c]{ll}%
-(x^{\prime}+\frac{v}{2}(T-s))^{2}, & x^{\prime}>-\frac{v}{2}(T-s),\\
0, & x^{\prime}\in \lbrack-v(T-s),-\frac{v}{2}(T-s)],\\
-(x^{\prime}+v(T-s))^{2}, & x^{\prime}<-v(T-s).
\end{array}
\right.
\]
For each $P_{t,x}^{\ast}\in \mathcal{P}_{t,x}^{\ast}$ defined in (\ref{e2-3-1}%
), we know that for, $s\in \lbrack t,T]$,%
\[
\int_{t}^{s}[\partial_{x}\phi(r,v\int_{t}^{r}\gamma_{\theta}d\theta
-1)v\gamma_{r}+\partial_{r}\phi(r,v\int_{t}^{r}\gamma_{\theta}d\theta
-1)]dr=0,\text{ }P_{t,x}^{\ast}\text{-a.s.,}%
\]
where $d\langle B\rangle_{r}=\gamma_{r}dr$. Then we obtain%
\[
\partial_{x}\phi(s,v\int_{t}^{s}\gamma_{\theta}d\theta-1)v\gamma_{s}%
+\partial_{s}\phi(s,v\int_{t}^{s}\gamma_{\theta}d\theta-1)=0,\text{ a.e. }%
s\in \lbrack t,T],\text{ }P_{t,x}^{\ast}\text{-a.s.}%
\]
By simple calculation of $\partial_{x}\phi$ and $\partial_{s}\phi$, we can
easily obtain that there exists a $\Omega_{0}\subset \Omega$ with
$P_{t,x}^{\ast}(\Omega_{0})=1$ such that, for $\omega \in \Omega_{0}$, a.e.
$s\in \lbrack t,T]$,%
\begin{equation}
\gamma_{s}(\omega)=\left \{
\begin{array}
[c]{ll}%
0.5, & v\int_{t}^{s}\gamma_{\theta}(\omega)d\theta-1>-\frac{v}{2}(T-s),\\
1, & v\int_{t}^{s}\gamma_{\theta}(\omega)d\theta-1<-v(T-s).
\end{array}
\right.  \label{ee3-40}%
\end{equation}
If $v\int_{t}^{s_{0}}\gamma_{\theta}(\omega)d\theta-1>-\frac{v}{2}(T-s_{0})$
for some $s_{0}\in(t,T]$, then $s^{\ast}:=\sup \{s<s_{0}:v\int_{t}^{s}%
\gamma_{\theta}(\omega)d\theta-1=-\frac{v}{2}(T-s)\}$ such that%
\[%
\begin{array}
[c]{l}%
v\int_{t}^{s^{\ast}}\gamma_{\theta}(\omega)d\theta-1=-\frac{v}{2}(T-s^{\ast
}),\\
v\int_{t}^{s}\gamma_{\theta}(\omega)d\theta-1>-\frac{v}{2}(T-s)\text{ for
}s\in(s^{\ast},s_{0}].
\end{array}
\]
By (\ref{ee3-40}), we get $\gamma_{s}(\omega)=0.5$ for $s\in(s^{\ast},s_{0}]$
a.e. From this, we have
\begin{align*}
v\int_{t}^{s}\gamma_{\theta}(\omega)d\theta-1  &  =v\int_{t}^{s^{\ast}}%
\gamma_{\theta}(\omega)d\theta-1+\frac{v}{2}(s-s^{\ast})\\
&  =-\frac{v}{2}(T-s^{\ast})+\frac{v}{2}(s-s^{\ast})\\
&  =-\frac{v}{2}(T-s),\text{ for }s\in(s^{\ast},s_{0}],
\end{align*}
which is a contradiction. Thus we obtain $v\int_{t}^{s}\gamma_{\theta}%
d\theta-1\leq-\frac{v}{2}(T-s)$ for $s\in \lbrack t,T]$. Similarly, we can get
$v\int_{t}^{s}\gamma_{\theta}d\theta-1\geq-v(T-s)$ for $s\in \lbrack t,T]$.
Taking $s=T$, we obtain%
\[
v\int_{t}^{T}\gamma_{\theta}d\theta-1=0,\text{ }P_{t,x}^{\ast}\text{-a.s.}%
\]
From this we get%
\[
v\int_{t}^{s}\gamma_{\theta}d\theta-1=-v\int_{s}^{T}\gamma_{\theta}d\theta
\in \lbrack-v(T-s),-\frac{v}{2}(T-s)].
\]
Thus%
\begin{equation}
\mathcal{P}_{t,x}^{\ast}=\{P\in \mathcal{P}:\langle B\rangle_{T}-\langle
B\rangle_{t}=v^{-1},\text{ }P\text{-a.s.}\}. \label{ee3-41}%
\end{equation}
For each $P_{t,x}^{\ast}\in \mathcal{P}_{t,x}^{\ast}$, it is clear that
$Y_{s}^{t,x,\bar{u}}=0$. But $V(s,X_{s}^{t,x,\bar{u}})=-\left(  \int_{t}%
^{s}(v\gamma_{\theta}-1)d\theta \right)  ^{2}$ may not be equal to $0$. Thus
the relational expression (\ref{ee3-1}) in Theorem \ref{MP-DPP-1} does not
hold for each $P_{t,x}^{\ast}\in \mathcal{P}_{t,x}^{\ast}$. Similar to the
analysis of (\ref{ee3-40}), we can obtain that $\mathcal{\tilde{P}}%
_{t,x}^{\ast}$ contains only one element $P_{t,x}^{\ast}$ such that
$v\gamma_{s}=1$, a.e. $s\in \lbrack t,T],$ $P_{t,x}^{\ast}$-a.s. Thus the
relational expression (\ref{ee3-1}) in Theorem \ref{MP-DPP-1} holds for
$P_{t,x}^{\ast}\in \mathcal{\tilde{P}}_{t,x}^{\ast}$. It is easy to verify that
the relational expressions (\ref{eee3-4}) and (\ref{ee3-3}) in Theorem
\ref{MP-DPP-1} hold for $P_{t,x}^{\ast}\in \mathcal{\tilde{P}}_{t,x}^{\ast}$.
\end{example}

\end{document}